\documentclass[11pt]{article}

\usepackage{amsfonts,amsmath,amssymb}
\usepackage{url}
\usepackage{epsfig,latexsym,graphicx}
\usepackage{esint}
\usepackage{graphicx,latexsym,amssymb,amsmath,amsfonts}
\newtheorem{maintheorem}{Theorem}

\newtheorem{theorem}{Theorem}[section]
\newtheorem{lemma}[theorem]{Lemma}

\newtheorem{proposition}[theorem]{Proposition}
\newtheorem{observation}[theorem]{Observation}

\newtheorem{definition}[theorem]{Definition}

\newenvironment{proof}{\noindent{\bf Proof.\,\ }}{\hfill\mbox{$\Box$}\smallskip}

\def\XXint#1#2#3{{\setbox0=\hbox{$#1{#2#3}{\int}$ }
\vcenter{\hbox{$#2#3$ }}\kern-.6\wd0}}

\newcommand{\E}{\mathbb{E}}
\renewcommand{\P}{\mathbb{P}}

\newcommand{\A}[1]{\mathcal{A}^{(#1)}}

\newcommand{\cd}{\mathcal{D}}
\newcommand{\ch}{\mathcal{H}}
\newcommand{\ce}{\mathcal{E}}
\newcommand{\cg}{\mathcal{G}}
\newcommand{\cj}{\mathcal{J}}
\newcommand{\cm}{\mathcal{M}}
\newcommand{\Z}{\mathbb{Z}}

\newcommand{\R}{\mathbb{R}}
\newcommand{\cf}{\mathcal{F}}
\newcommand{\cc}{\stackrel {c,c}{\longleftrightarrow}}
\newcommand{\cs}{\stackrel {c,s}{\longleftrightarrow}}
\newcommand{\scc}{\stackrel {s,c}{\longleftrightarrow}}
\newcommand{\ssc}{\stackrel {s,s}{\longleftrightarrow}}

\newcommand{\cso}{\stackrel {c,s,*}{\longleftrightarrow}}
\newcommand{\scco}{\stackrel {s,c,*}{\longleftrightarrow}}
\newcommand{\ssco}{\stackrel {s,s,*}{\longleftrightarrow}}

\title{Scheduling of non-colliding random walks}

\author{R. Basu \thanks{Department of Statistics, University of California, Berkeley.
Email: riddhipratim@stat.berkeley.edu. Supported by UC Berkeley Graduate Fellowship.}
\and
V. Sidoravicius
\thanks{ IMPA, Estrada Dona Castorina 110, Rio de Janeiro, Brasil. Email: vladas@impa.br}
\and
A. Sly
\thanks{Department of Statistics, University of California, Berkeley and MSI, Australian National University.  Email: sly@stat.berkeley.edu
Supported in part by DMS-1208338, DMS-1352013 and Sloan Fellowship.
}}

\begin{document}
\date{}
\maketitle

\begin{abstract}
On the complete graph
${\cal{K}}_M$ with $M \ge3$ vertices  consider two independent discrete  time
random walks $\mathbb{X}$ and $\mathbb{Y}$,
choosing their steps uniformly at random.
A pair of
trajectories   $\mathbb{X} = \{ X_1, X_2, \dots \}$ and
$\mathbb{Y} = \{Y_1, Y_2, \dots \}$ is called {\it{non-colliding}}, if
by delaying
their jump times one can keep both walks at distinct vertices forever.
It was conjectured by P. Winkler that for large enough $M$
the set  of pairs of non-colliding trajectories $\{\mathbb{X},\mathbb{Y} \} $
has positive measure. N. Alon translated this problem to the language
of coordinate percolation, a class of dependent percolation models,
which in most situations is not tractable by methods of Bernoulli percolation. In this representation
Winkler's conjecture is equivalent to the existence of an infinite open cluster for large enough $M$.   In this
paper we establish the conjecture building upon
the renormalization techniques developed in \cite{BS14}.

\end{abstract}

\section{Introduction}

Bernoulli percolation has been a paradigm model for spatial randomness for last half
a century.  The deep and rich understanding that emerged is a celebrated success story
of contemporary probability. In the mean time several natural
questions arising from mathematical physics and theoretical computer science has necessitated the
study of models containing more complicated dependent structures, which are not amenable
to the tools of Bernoulli percolation. Among them we could mention classical gas of interacting Brownian paths \cite{Sm}, loop soups \cite{RW} and random interlacements \cite{C, Sz}. A particular subclass of models that has received attention is a class of "coordinate percolation" models, which were introduced, motivated by problems of statistical physics,  in late eighties by B. T\' oth under the name "corner percolation", later studied in  \cite{pete2008}, and in early nineties in theoretical computer science by P. Winkler, later studied in several its variants in  \cite{CTW:00,Winkler:00,MW08,BW09}. Problems of embedding one random sequence into another can also be cast into this framework (\cite{GLR:10, BS14, Grimmett:09, Gacs:04, Gacs:12, KLSV}), which in turn is intimately related to quasi-isometries of random objects \cite{Peled:10, BS14}.

In this work we focus on one particular model in this class, introduced by Winkler, which in its original formulation relates to clairvoyant scheduling of two independent random walks on a complete graph. More precisely, on the complete graph ${\cal{K}}_M$ with $ M \ge3$ vertices consider two independent discrete time random walks $\mathbb{X}$ and $\mathbb{Y}$ which move by choosing  steps uniformly at random. Two trajectories (realizations)   $\mathbb{X} = \{ X_1, X_2, \dots \}$ and
$\mathbb{Y} = \{Y_1, Y_2, \dots \}$ are called {\it{non-colliding}}, if, knowing all steps of $\mathbb{X}$ and $\mathbb{Y}$, one can keep both walks on distinct vertices forever by delaying their jump-times  appropriately. The question of interest here is whether the set of non-colliding pairs of trajectories have positive probability. For $M=3$ the measure of non-colliding pairs is zero (see Corollary 3.4 \cite{Winkler:00}).
It was conjectured by P. Winkler \cite{CTW:00} that for large enough $M$, in particular it is believed for $M \ge 4$ based on simulations, the
set  of non-colliding trajectories $\{\mathbb{X},\mathbb{Y} \} $
has positive measure. The question became prominent as the {\it{clairvoyant demon problem}}.

N. Alon translated this problem into the language of coordinate percolation. Namely,
let $\mathbb{X}=(X_1,X_2,\ldots)$ and $\mathbb{Y}=(Y_1,Y_2,\ldots)$ be two i.i.d. sequences with
$$
\mathbb{P}(X_i=k)=\mathbb{P}(Y_j=k)=\frac{1}{M} \; \text{for } \; k=1,2,\ldots ,M \text{ and for } i,j=1,2,\ldots .
$$
Define an oriented percolation process on $\Z_+ \times \Z_+$: the vertex  $(i_1,i_2)\in \Z^2_{> 0}$ will be called ``closed'' if $X_{i_1}=Y_{i_2}$. Otherwise it is called ``open''. It is curious to notice that this percolation process (for $M$=2)  was introduced much earlier by Diaconis and Freedman \cite{DF81} in the completely different context of studying visually distinguishable random patterns in connection with Julesez's conjecture. It is easy to observe that a pair of trajectories $\{\mathbb{X},\mathbb{Y}\}$ is non-colliding if and only if there is an open oriented infinite path starting at the vertex $(1,1)$. The issue of settling Winkler's conjecture then translates to proving that for $M$ sufficiently large, there is percolation with positive probability, which is our main result in this paper.  For $\mathbb{X}$ and $\mathbb{Y}$ as above, we say $\mathbb{X}\longleftrightarrow \mathbb{Y}$ if there exists an infinite open oriented  path starting from $(1,1)$.

\begin{maintheorem}
\label{t:main}
For all $M$ sufficiently large, $\mathbb{P}(\mathbb{X} \longleftrightarrow \mathbb{Y})>0$, thus clairvoyant scheduling is possible.
\end{maintheorem}

\medskip

\subsection{Related Works}
This scheduling problem first appeared in the context of distributed computing \cite{CTW:00} where it is shown that two independent random walks on a finite connected non-bipartite graph will collide in a polynomial time even if a scheduler tries to keep them apart, unless the scheduler is clairvoyant. In a recent work \cite{Ang13}, instead of independent random walks, by allowing coupled random walks, it was shown that a large number of random walks can be made to avoid one another forever.  In the context of clairvoyant scheduling of two independent walks, the non-oriented version of the oriented percolation process described above was studied independently in \cite{Winkler:00} and \cite{BBS:00} where they establish that in the non-oriented model there is percolation with positive probability if and only if $M\geq 4$. In \cite{Gacs:00} it was established that, if there is percolation, the chance that the cluster dies out after reaching distance $n$ must decay polynomially in $n$, which showed that, unlike the non-oriented models, this model was fundamentally different from Bernoulli percolation, where such decay is exponential.

In \cite{BS14} a multi-scale structure was developed to tackle random embedding problems which can be recast in co-ordinate percolation framework. As a corollary of a general embedding theorem, it was proved there that an i.i.d. Bernoulli sequence can almost surely be embedded into another in a Lipschitz manner provided that the Lipschitz constant is sufficiently large. It also led to a proof of rough isometry of two one-dimensional Poisson processes as well as a new proof of Winkler's compatible sequence problem. In this work we build upon the methods of \cite{BS14}, using a similar multi-scale structure, but with crucial adaptations. An earlier proof of Theorem 1 appeared in \cite{Gacs:11} with a very difficult multi-scale argument. Our proof is different and we believe gives a clearer inductive structure. We also believe that our proof can be adapted to deal with this problem on several other graphs, as well as in the case where there are multiple random walks.

\bigskip

\subsection{Outline of the proof}\label{s:outline}
Our proof relies on multi-scale analysis. The key idea is to divide the original sequences into blocks of doubly exponentially growing length scales
$L_j=L_0^{\alpha^j}$, for $j\geq 1$, and at each of these levels $j$ we have a definition of a ``good'' block. The multi-scale structure that we construct has a number of parameters, $\alpha,\beta,\delta, m, k_0, R$ and $L_0$
which must satisfy a number of relations described in the next subsection. Single characters in the original sequences $\mathbb{X}$ and $\mathbb{Y}$ constitute the level 0 blocks.

Suppose that we have constructed the blocks up to level $j$ denoting the sequence of blocks of level $j$ as $(X_1^{(j)},X_2^{(j)}\ldots)$.
In Sect.~\ref{s:prelim} we give a construction of $(j+1)$-level blocks out of $j$-level sub-blocks in such way that the blocks are independent and, apart from the first block, identically distributed. Construction of blocks at level 1 has slight difference from the general construction.

At each level we have a definition which distinguishes some of the blocks as good.  This is designed in such a manner that at each level, if we look at  the rectangle in the lattice determined by a good block $X$ and a random block $Y$\!,  then, with high probability,  it will have many  open paths with varying slopes through it. For a precise definition see Definitions~\ref{d:cs} and~\ref{d:ss}.
Having these paths with different slopes  will help achieve improving estimates of the probability of the event of having a path from the bottom left corner to the top right
corner of the lattice rectangle determined by random blocks $X$ and $Y$,  denoted by $[X\cc Y]$, at higher levels.

The proof then involves a series of recursive estimates at each level, given in Sect. \ref{s:recursive}.  We require that at level $j$ the probability of a block being good is at least $1-L_j^{-\delta}$, so that the vast majority of blocks are good.  Furthermore, we obtain tail bounds on $\P(X\cc Y \mid X)$ by showing that for $0<p\leq \frac{3}{4}+2^{-(j+3)}$,
\[
\P(\P(X\cc Y \mid X)\leq p)\leq p^{m+2^{-j}}L_j^{-\beta},
\]
where $\beta$ and $m$ are parameters mentioned at the beginning of this section.
We show the similar bound for $\mathbb{Y}$-blocks as well.
We also ask that  the length of blocks satisfy an exponential tail estimate. The full inductive step is given in Sect. \!\ref{induction}.  Proving this constitutes the main work of the paper.

We use the key quantitative estimate provided by Lemma~\ref{l:totalSizeBound}  which is taken from \cite{BS14} (see Lemma 7.3, \cite{BS14}), which bounds the probability of a block having: {\it{a}}) an excessive length,  {\it{b}}) too many bad sub-blocks,  {\it{c}}) a particularly difficult collection of sub-blocks, where we quantify the difficulty of a collection of bad sub-blocks $\{X_i \}_{i=1}^k$ by the  value of $\prod_{i=1}^k   \P [X_i \cc Y| X]$, where $Y$ is a random block at  the same level. In order to achieve the improvement on the tail bounds of $\P(X\cc Y \mid X)$ at each level, we take advantage of the flexibility in trying a large number of potential positions
to cross the rectangular strips determined by each member of a small collection of bad sub-blocks, obtained by using the recursive estimates on probabilities of existence of paths of varying slopes through rectangles determined by collections of good sub-blocks.

To this effect we also borrow the notion of generalised mappings developed in \cite{BS14} to describe such potential mappings. Our analysis is split  into 5 different cases. To push through the estimate of the probability of having many open paths of varying slopes at a higher level, we make some finer geometric constructions.   
To complete the proof we note that  $X_1^{(j)}$ and $Y_1^{(j)}$ are good for all $j$ with positive probability.  Using the definition of good blocks and a compactness argument we conclude the existence of an infinite open path with positive probability.

\subsection{Parameters}\label{s:parameters}
Our proof involves a collection of parameters $\alpha,\beta,\delta,k_0,m$ and $R$ which must satisfy a system of constraints.  The required constraints are
\[
\alpha>6, \delta>2\alpha \vee 48, \beta>\alpha(\delta+1), m>9\alpha\beta, k_0> 36\alpha\beta, R> 6(m+1).
\]
To fix on a choice we will set
\begin{equation}\label{e:parameters}
\alpha=10, \delta=50, \beta=600, m=60000, k_0=300000, R=400000.
\end{equation}
Given these choices we then take $L_0$ to be a sufficiently large integer.  We did not make a serious attempt to optimize the parameters or constraints, sometimes for the sake of clarity of exposition.

\section{The Multi-scale Structure}\label{s:prelim}

Our strategy for the proof of Theorem~\ref{t:main} is to partition the sequences $\mathbb{X}$ and $\mathbb{Y}$ into blocks at each level $j\geq 1$. For each $j\geq 1$, we write $\mathbb{X}=(X_1^{(j)},X_2^{(j)},\ldots)$ where we call each $X_i^{(j)}$ a level $j$ $\mathbb{X}$-block, similarly we write $\mathbb{Y}=(Y_1^{(j)},Y_2^{(j)},\ldots)$.  Most of the time we would clearly state that something is a level $j$ block and drop the superscript $j$. Each of the $\mathbb{X}$-block (resp. $\mathbb{Y}$-block) at level $(j+1)$ is a concatenation of a number of level $j$ $\mathbb{X}$-blocks,  where the level $0$ blocks are just the elements of the original sequence.

\subsection{Recursive Construction of Blocks}

 Level $1$ blocks are constructed inductively as follows:

Suppose the first $k$ blocks $X_1^{(1)},\ldots , X_k^{(1)}$ at level $1$ have already been constructed and suppose that the rightmost element of $X_k^{(1)}$ is $X_{n_k}^{(0)}$. Then $X_{{n_k}+1}^{(1)}$ consists of the elements $X_{{n_k}+1}^{(0)},X_{{n_k}+2}^{(0)},\ldots ,X_{{n_k}+l}^{(0)}$ where
\begin{equation}
\label{e:levelonex}
l=\min\{t\geq L_1: X_{{n_k}+t}^{(0)}=1~\mbox{mod}~4~\text{and}~ X_{{n_k}+t+1}^{(0)}=0~\mbox{mod}~4\}.
\end{equation}
The same definition holds for $k=0$, assuming $n_0 = -1$.
Recall that $L_1=L_0^{\alpha}$.

Similarly, suppose the first $k$  $\mathbb{Y}$-blocks at level $1$ are $Y_1^{(1)},\ldots , Y_k^{(1)}$ and also suppose that the rightmost element of $Y_k^{(1)}$ is $Y_{n_k}^{(0)}$. Then $Y_{k+1}^{(1)}$ consists of the elements $Y_{{n_k}+1}^{(0)},Y_{{n_k}+2}^{0)},\ldots ,$ $Y_{{n_k}+l}^{(0)}$ where

\begin{equation}
\label{e:leveloney}
l=\min\{t\geq L_1: Y_{{n_k}+t}^{(0)}=3~\mbox{mod}~4~\text{and}~ Y_{{n_k}+t+1}^{(0)}=2~\mbox{mod}~4\}.
\end{equation}

We shall denote the length of an $\mathbb{X}$-block $X$ (resp. a $\mathbb{Y}$-block $Y$) at level $1$ by $L_X=L_1+T_X^{(1)}$ (resp. $L_Y=L_1+T_Y^{(1)}$).
Notice that this construction, along with Assumption 1, ensures that the blocks at level one are independent and identically distributed.

At each level $j\geq 1$, we also have a recursive definition of ``$good$" blocks (see Definition~\ref{d:good}). Let $G_j^{\mathbb{X}}$ and $G_j^{\mathbb{Y}}$ denote the set of good $\mathbb{X}$-blocks and  good $\mathbb{Y}$-blocks at $j$-th level respectively. Now we are ready to describe the recursive construction of the blocks $X_i^{(j)}$ and $Y_i^{(j)}$ for $j\geq 2$.

The construction of blocks at level $j\geq 2$ is similar for both $\mathbb{X}$ and $\mathbb{Y}$ and we only describe the procedure to form the blocks for the sequence $\mathbb{X}$. Let us suppose we have already constructed the blocks of partition  up to level $j$ for some $j\geq 1$ and we have $X=(X_1^{(j)},X_2^{(j)},\ldots)$. Also assume we have defined the ``good" blocks at level $j$, i.e., we know $G_j^{\mathbb{X}}$.
We describe how to partition $\mathbb{X}$ into level $(j+1)$ blocks: $\mathbb{X}=(X_1^{(j+1)},X_2^{(j+1)},\ldots)$.

Suppose the first $k$ blocks $X_1^{(j+1)},\ldots , X_k^{(j+1)}$ at level $(j+1)$ has already been constructed and suppose that the rightmost level $j$-subblock of $X_k^{(j+1)}$ is $X_m^{(j)}$. Then $X_{k+1}^{(j+1)}$ consists of the sub-blocks $X_{m+1}^{(j)},X_{m+2}^{(j)},\ldots ,X_{m+l+L_j^3}^{(j)}$ where $l>L_j^3+L_j^{\alpha-1}$ is selected in the following manner. Let $W_{k+1,j+1}$ be a geometric random variable having $\mbox{Geom}(L_j^{-4})$ distribution and independent of everything else. Then
$$l=\min\{s\geq L_j^3+L_j^{\alpha -1}+ W_{k+1,j+1}: X_{m+s+i}\in G_{j}^{\mathbb{X}} ~\text{for}~1\leq i \leq 2L_j^3 \}.$$
That such an $l$ is finite with probability 1 will follow from our recursive estimates. The case $k=0$ is dealt with as before.

\medskip

\noindent Put simply, our block construction mechanism at level $(j+1)$ is as follows:\\
\emph{Starting from the right boundary of the previous block, we include $L_j^3$ many sub-blocks, then further $L_j^{\alpha-1}$ many sub-blocks, then a  $\mbox{Geom}(L_j^{-4})$ many sub-blocks. Then we wait for the first occurrence of a run of $2L_j^3$ many consecutive good sub-blocks, and end our block at the midpoint of this run.}


We now record two simple but useful properties of the blocks thus constructed in the following observation. Once again a similar statement holds for $\mathbb{Y}$-blocks.

\begin{observation}\label{o:blockStructure}
Let $\mathbb{X}=(X_1^{(j+1)},X_2^{(j+1)},\ldots)=(X_1^{(j)}, X_2^{(j)}, \ldots)$ denote the partition of $\mathbb{X}$ into blocks at levels $(j+1)$ and $j$ respectively. Then the following hold.

\begin{enumerate}
\item Let $X_i^{(j+1)}=(X_{i_1}^{(j)},X_{i_1+1}^{(j)},\ldots X_{i_1+l}^{(j)})$. For $i\geq 1$, $X_{i_1+l+1-k}^{(j)}\in G_j^{\mathbb{X}}$ for each $k$, $1\leq k \leq L_j^3$. Further, if $i>1$, then $X_{i_1+k-1}^{(j)}\in G_j^{\mathbb{X}}$ for each $k$, $1\leq k \leq L_j^3$. That is, all blocks at level $(j+1)$, except possibly the leftmost one, $X_1^{(j+1)}$, are guaranteed to have at least $L_j^3$ ``good" level $j$ sub-blocks at either end. Even $X_1^{(j+1)}$ ends in $L_j^3$ many good sub-blocks.
\item The blocks $X_1^{(j+1)},X_2^{(j+1)}, \ldots $ are independently distributed. In fact, $X_2^{(j+1)}, X_3^{(j+1)}, \ldots$ are independently and identically distributed according to some law, say $\mu_{j+1}^{\mathbb{X}}$. Furthermore, conditional on the event $\{X_i^{(k)}\in G_k^{\mathbb{X}} ~\text{for}~ i=1,2,\ldots, L_k^3,~\text{for all}~k\leq j\}$, the $(j+1)$-th level blocks $X_1^{(j+1)}, X_2^{(j+1)},\ldots$ are independently and identically distributed according to the law $\mu_{j+1}^{\mathbb{X}}$.
\end{enumerate}
\end{observation}

From now on whenever we say ``a (random) $\mathbb{X}$-block at level $j$", we would imply that it has law $\mu_{j}^{\mathbb{X}}$, unless explicitly stated otherwise. Similarly let us denote the corresponding law of ``a (random) $\mathbb{Y}$-block at level $j$" by $\mu_{j}^{\mathbb{Y}}$.

Also, for $j> 0$, let $\mu_{j,G}^{\mathbb{X}}$ denote the conditional law of an $\mathbb{X}$ block at level $j$, given that it is in $G_j^{\mathbb{X}}$. We define $\mu_{j,G}^{\mathbb{Y}}$ similarly.

We observe that we can construct a block with law $\mu_{j+1}^{\mathbb{X}}$ (resp. $\mu_{j+1}^{\mathbb{Y}}$) in the following alternative manner without referring to the the sequence $\mathbb{X}$ (resp. $\mathbb{Y}$):
\begin{observation}\label{o:blockRepresentation}
Let $X_1,X_2,X_3,\ldots$ be a sequence of independent level $j$ $\mathbb{X}$-blocks such that $X_i\sim \mu_{j,G}^{\mathbb{X}}$ for $1\leq i \leq L_j^3$ and $X_i\sim \mu_{j}^{\mathbb{X}}$ for $i> L_j^3$. Now let $W$ be a $Geom(L_j^{-4})$ variable independent of everything else. Define as before
$$l=\min\{i\geq L_j^3+L_j^{\alpha -1}+W: X_{i+k}\in G_j^{\mathbb{X}} ~\text{for}~1\leq k\leq 2L_j^3\}.$$
Then $X=(X_1,X_2,\ldots ,X_{l+L_j^3})$ has law $\mu_{j+1}^{\mathbb{X}}$.
\end{observation}

Whenever we have a sequence $X_1,X_2,...$ satisfying the condition in the observation above, we shall call $X$  the (random) level $(j+1)$ block
constructed from $X_1,X_2,....$ and we shall denote the corresponding geometric variable by $W_X$ and set $T_X=l-L_j^3-L_j^{\alpha -1}$.

We still need to define good blocks, to complete the structure, we now move towards that direction.

\subsection{Corner to Corner, Corner to Side and Side to Side Mapping probabilities}

Now we make some definitions that we are going to use throughout our proof.
Let $X=(X_{s+1}^{(j)},X_{s+2}^{(j)},\ldots ,X_{s+l_X}^{{(j)}})=(X_{a_1}^{(0)},\ldots ,X_{a_2}^{(0)})$ be a level $(j+1)$ $\mathbb{X}$-block ($j\geq 1$) where $X_i^{(j)}$'s and $X_i^{(0)}$ are the level $j$ sub-blocks and the level 0 sub-blocks constituting it respectively. Similarly let $Y=(Y_{s'+1}^{(j)},Y_{s'+2}^{(j)},...,Y_{s+l_Y}^{(j)})=(Y_{b_1}^{(0)},\ldots ,Y_{b_2}^{(0)})$ is a level $(j+1)$ $\mathbb{Y}$-block. Let us consider the lattice rectangle $[a_1,a_2]\times [b_1,b_2] \cap \Z ^2$, and denote it by $X\times Y$. It follows from (\ref{e:levelonex}) and (\ref{e:leveloney}) that sites at all the four corners of this rectangle are open.

\begin{definition}[Corner to Corner Path]
We say that there is a corner to corner path in $X\times Y$, denoted by $$ X\cc Y,$$ if there is an open oriented path in $X\times Y$ from $(a_1,b_1)$ to $(a_2,b_2)$.
\end{definition}

\noindent A site $(x,b_2)$ and respectively a site $(a_2,y)$, on the top, respectively on the right side of $X\times Y$,  is called "reachable from bottom left site" if there is an open oriented path in $X\times Y$ from $(a_1,b_1)$ to that site.

\medskip

\noindent Further, the intervals $[a_1,a_2]$ and $[b_1,b_2]$ will be partitioned into ``chunks" $\{C_r^{X}\}_{r\geq 1}$ and $\{C_r^{Y}\}_{r\geq 1}$  respectively in the following manner. Let for any $\mathbb{X}$-block $\tilde{X}$ at any level $j\geq 1$, $$\mathcal{I}(\tilde{X})=\{a\in \mathbb{N}: \tilde{X} ~\text{contains the level 0 block}~  X_{a}^{(0)}\}.$$

\noindent Let $X=(X_{s+1}^{(j)},X_{s+2}^{(j)},\ldots ,X_{s+l_X}^{(j)})$, and $n_X := \lfloor l_X/L_j^4 \rfloor$. Similarly we define
$n'_Y := \lfloor l_Y/L_j^4 \rfloor$.

\begin{definition}[Chunks]  The discrete segment $C_k^{X} \subset \mathcal{I}(X)$ defined as
\begin{equation}
C_k^{X} := \begin{cases} \cup_{t=(k-1)L_j^4+1}^{kL_j^4}\mathcal{I}(X_{s+t}^{(j)}), \; & k= 1, \dots, n_X -1; \\
\cup_{t=k_X L_j^4+1}^{l_X}\mathcal{I}(X_{s+t}^{(j)}), \; & k= n_X;
\end{cases}
\end{equation}
is called the $k^{th}$\,chunk of $X$.
\end{definition}
By  $C^X$ and $C^Y$ we denote the set of all chunks $\{C_k^{X}\}_{k=1}^{n_X}$ and $\{C_k^{Y}\}_{k=1}^{n'_Y} $ of $X$
and $Y$ respectively. In what follows the letters $\cal{T},\cal{B},\cal{L},\cal{R}$ will stand for "top", "bottom", "left", and "right"\!, respectively. Define:
 \begin{align}
 C_{\cal{B}}^X=C^X\times \{1\}, \quad & C_{\cal{T}}^X=C^X\times \{n'_Y\}, \notag \\
  C_{\cal{L}}^Y=\{1\}\times C^Y, \quad & C_{\cal{R}}^Y=\{n_X\} \times C^Y. \notag
\end{align}

\begin{definition}[Entry/Exit Chunk, Slope Conditions] A pair
 $(C_k^X,1)\in C_{\cal{B}}^X$, $k\in [L_j, n_X-L_j]$ is called an entry chunk (from the bottom) if it satisfies the slope condition
\begin{equation}
\label{e:slopeentrychunkbottom}
\frac{1-2^{-(j+4)}}{R}\leq \frac{n_Y'-1}{n_X-k}\leq R(1+2^{-(j+4)}).
\end{equation}
Similarly,  $(1,C_k^Y)\in C_{\cal{L}}^Y$, $k\in [L_j, n'_Y-L_j]$, is called an entry chunk (from the left) if  it satisfies the slope condition

\begin{equation}
\label{e:slopeentrychunkleft}
\frac{1-2^{-(j+4)}}{R}\leq \frac{n_Y'-k}{n_X-1}\leq R(1+2^{-(j+4)}).
\end{equation}
The set of all entry chunks is denoted by $\mathcal{E}_{in}(X,Y)\subseteq (C_{\cal{B}}^X\cup C_{\cal{L}}^Y)$. The set of all exit chunks $\mathcal{E}_{out}(X,Y)$ is defined in a similar fashion.

We call $(e_1,e_2)\in (C_B^X\cup C_L^Y) \times (C_T^X\cup C_R^Y)$ is an "entry-exit pair of chunks" if the following conditions are satisfied. Without loss of generality assume $e_1=(C_{k}^X,1)\in C_B^X$ and $e_2=(n'_X,C_{k'}^{Y})\in C_R^Y$. Then  $(e_1,e_2)$ is called an "entry-exit pair" if $k\in [L_j, n_X-L_j]$, $k'\in [L_j,n'_Y-L_j]$ and they satisfy the slope condition

\begin{equation}
\label{e:slopeentryexit}
\frac{1-2^{-(j+4)}}{R}\leq \frac{k'-1}{n_X-k}\leq R(1+2^{-(j+4)}).
\end{equation}
Let us denote the set of all "entry-exit pair of chunks" by $\mathcal{E}(X,Y)$.
\end{definition}

\begin{definition}[Corner to Side and Side to Corner Path]
\label{d:cs}
We say that there is a corner to side path in $X\times Y$, denoted by $$ X\cs Y$$ if for each $(C_k^{X},n_X), (n'_Y,C_{k'}^Y)\in \mathcal{E}_2(X,Y)$

$$\#\{a\in C_k^{X}: (a,b_2)~\text{is reachable from $(a_1,b_1)$ in $X\times Y$ }\}\geq \left(\frac34 +2^{-(j+5)}\right)|C_k^X|,$$


$$\#\{b\in C_k^Y: (a_2,b)~\text{is reachable from $(a_1,b_1)$ in $X\times Y$ }\}\geq \left(\frac34 +2^{-(j+5)}\right)|C_k^Y|.$$
\end{definition}

Side to corner paths in $X\times Y$, denoted $X\scc Y$ is defined in the same way except that in this case we want paths from the bottom or left side of the rectangle $X\times Y$ to its top right corner and use $\mathcal{E}_1(X,Y)$ instead of $\mathcal{E}_2(X,Y)$.


\textbf{Condition S:} Let $(e_1,e_2)\in \mathcal{E}(X,Y)$. Without loss of generality we assume $e_1=(C_{k_1}^{X},1)\in C_B^X$ and $e_2=(n_X, C_{k_2}^{Y})\in C_R^Y$. $(e_1,e_2)$ is said to satisfy condition $S$ if there exists $A\subseteq C_{k_1}^{X}$ with $|A|\geq \left(\frac34 +2^{-(j+5)}\right)|C_{k_1}^X|$ and $B\subseteq C_{k_2}^Y$ with $|B|\geq \left(\frac34 +2^{-(j+5)}\right)|C_{k_2}^Y|$  such that for all $a\in A$ and for all $b\in B$ there exist an open path in $X\times Y$ from $(a,b_1)$ to $(a_2,b)$. Condition $S$ is defined similarly for the other cases.

\begin{definition}[Side to Side Path]
\label{d:ss}
We say that there is a side to side path in $X\times Y$, denoted by $$ X\ssc Y$$ if each $(e_1,e_2)\in \mathcal{E}(X,Y)$ satisfies condition $S$.
\end{definition}

It will be convenient for us to define corner to corner, corner to side, and side to side paths not only in rectangles determined by one $\mathbb{X}$-block and one $\mathbb{Y}$-block. Consider a $j+1$-level $\mathbb{X}$-block $X=(X_1,X_2,\ldots ,X_n)$ and a $j+1$-level $\mathbb{Y}$ block $Y=(Y_1,\ldots, Y_{n'})$ where $X_i,Y_i$ are $j$ level subblocks constituting it.  Let $\tilde{X}$ (resp. $\tilde{Y}$) denote a sequence of consecutive sub-blocks of $X$ (resp. $Y$), e.g., $\tilde{X}=(X_{t_1},X_{t_1+1},\ldots , X_{t_2})$ for $1\leq t_1\leq t_2\leq n$. Call $\tilde{X}$ to be a \emph{segment} of $X$. Let $\tilde{X}=(X_{t_1},X_{t_1+1},\ldots , X_{t_2})$ be a segment of $X$ and let $\tilde{Y}=(Y_{t'_1},Y_{t'_1+1},\ldots , Y_{t'_2})$ be a segment of $Y$. Let $\tilde{X}\times \tilde{Y}$ denote the rectangle in $\Z^2$ determined by $\tilde{X}$ and $\tilde{Y}$. Also let $X_{t_1}=(X_{a_1}^{(0)},\ldots, X_{a_2}^{(0)})$, $X_{t_2}=(X_{a_3}^{(0)},\ldots, X_{a_4}^{(0)})$, $Y_{t'_1}=(Y_{b_1}^{(0)},\ldots, Y_{a_2}^{(0)})$, $Y_{t'_2}=(Y_{b_3}^{(0)},\ldots, Y_{b_4}^{(0)})$.

\begin{enumerate}
\item[$\bullet$] We denote by $\tilde{X}\cc \tilde{Y}$, the event that there exists an open oriented path from the bottom left corner to the top right corner of $\tilde{X}\times \tilde{Y}$.
\item[$\bullet$] Let $\tilde{X}\cso \tilde{Y}$ denote the event that
$$\left\{\#\{b\in [b_3,b_4]: (a_2,b)~\text{is reachable from}~(a_1,b_1)\}\geq (\frac{3}{4}+2^{-(j+7/2)})(b_4-b_3)\right\}~\text{and}$$
$$\left\{\#\{a\in [a_3,a_4]: (a,b_4)~\text{is reachable from}~(a_1,b_1)\}\geq (\frac{3}{4}+2^{-(j+7/2)})(a_4-a_3)\right\}.$$
$\tilde{X}\scco \tilde{Y}$ is defined in a similar manner.
\item[$\bullet$] We set $\tilde{X}\ssco \tilde{Y}$ to be the following event. There exists $A \subseteq [a_1,a_2]$ with $|A|\geq (\frac{3}{4}+2^{-(j+7/2)})(a_2-a_1)$, $A' \subseteq [a_3,a_4]$ with $|A'|\geq (\frac{3}{4}+2^{-(j+7/2)})(a_4-a_3)$, $B \subseteq [b_1,b_2]$ with $|B|\geq (\frac{3}{4}+2^{-(j+7/2)})(b_2-b_1)$ and $B' \subseteq [b_3,b_4]$ with $|B|\geq (\frac{3}{4}+2^{-(j+7/2)})(b_4-b_3)$ such that for all $a\in A, a'\in A'$, $b\in B$, $b'\in B'$ we have that $(a_4,b')$ and $(a',b_4)$) are reachable from $(a,b_1)$ and $(a_1,b)$.
\end{enumerate}

\begin{definition}[Corner to Corner Connection probability]
For $j\geq 1$, let $X$ be an $\mathbb{X}$-block at level $j$ and let $Y$ be a $\mathbb{Y}$-block at level $j$. We define the corner to corner connecting probability of $X$ to be  $S_j^{\mathbb{X}}(X)=\mathbb{P}(X\cc Y|X)$. Similarly we define $S_j^{\mathbb{Y}}(Y)=\mathbb{P}(X\cc Y|Y)$.
\end{definition}

As noted above the law of $Y$ is $\mu_j^{\mathbb{Y}}$ in the definition of $S_j^{\mathbb{X}}$ and the law of $X$ is $\mu_j^{\mathbb{X}}$ in the definition of $S_j^{\mathbb{Y}}$.
%


%
%

\subsection{Good blocks}\label{s:goodblocks}
To complete the description, we need to give the definition of ``good'' blocks at level $j$ for each $j\geq 1$ which we have alluded to above. With the definitions from the preceding section, we are now ready to give the recursive definition of a ``good" block as follows.  As usual we only give the definition for $\mathbb{X}$-blocks, the definition for $\mathbb{Y}$ is similar.\\

Let $X^{(j+1)}=(X_1^{(j)},X_2^{(j)},\ldots, X_n^{(j)})$ be an $\mathbb{X}$ block at level $(j+1)$. Notice that we can form blocks at level $(j+1)$ since we have assumed that we already know $G_j^{\mathbb{X}}$.

\begin{definition}[Good Blocks]
\label{d:good}
We say $X^{(j+1)}$ is a good block at level $(j+1)$ (denoted $X^{(j+1)}\in G_{j+1}^{\mathbb{X}}$) if the following conditions hold.

\begin{enumerate}
\item[(i)] It starts with $L_j^3$ good sub-blocks, i.e., $X_i^{(j)}\in G_j^{\mathbb{X}}$ for $1\leq i \leq L_j^3$. (This is required only for $j>0$, as there are no good blocks at level $0$ this does not apply for the case $j=0$).

\item[(ii)] $\mathbb{P}(X \ssc Y|X)\geq 1-L_{j+1}^{-2\beta}$

\item[(iii)] $\mathbb{P}(X \cs Y|X)\geq 9/10 + 2^{-(j+4)}$ and $\mathbb{P}(X \scc Y|X)\geq 9/10 + 2^{-(j+4)}.$

\item[(iv)] $S_j^{\mathbb{X}}(X)\geq 3/4 + 2^{-(j+4)}$.

\item[(v)] The length of the block satisfies
$n\leq L_{j}^{\alpha-1}+L_j^5$.
\end{enumerate}
\end{definition}

%

\section{Recursive estimates}\label{s:recursive}

Our proof of the theorem depends on a collection of recursive estimates, all of which are proved together by induction. In this section we list these estimates for easy reference. The proof of these estimates are provided in the next few sections. We recall that for all $j>0$ $L_j=L_{j-1}^{\alpha}=L_0^{\alpha ^j}$.

\subsection{Tail Estimate}

\begin{enumerate}
\item[I.]
Let $j\geq 1$. Let $X$ be a $\mathbb{X}$-block at level $j$ and  let $m_j=m+2^{-j}$. Then
\begin{equation}
\label{tailx1}
\mathbb{P}(S_j^{\mathbb{X}}(X)\leq p)\leq p^{m_j}L_j^{-\beta}~~\text{for}~~p\leq \frac 34 + 2^{-(j+3)}.
\end{equation}
Let $Y$ be a $\mathbb{Y}$-block at level $j$. Then
\begin{equation}
\label{taily1}
\mathbb{P}(S_j^{\mathbb{Y}}(Y)\leq p)\leq p^{m_j}L_j^{-\beta}~~\text{for}~~p\leq \frac 34 + 2^{-(j+3)}.
\end{equation}

\end{enumerate}

%
%
\subsection{Length Estimate}
\begin{enumerate}
\item[II.]
For $X$ an $\mathbb{X}$-block at at level $j\geq 0$,
\begin{equation}
\label{lengthx}
\mathbb{E}[\exp (L_{j-1}^{-6}(|X|-(2-2^{-j})L_j))] \leq 1.
\end{equation}

Similarly for $Y$, a $\mathbb{Y}$-block at level $j$, we have
\begin{equation}
\label{lengthy}
\mathbb{E}[\exp (L_{j-1}^{-6}(|Y|-(2-2^{-j})L_j))] \leq 1.
\end{equation}
\end{enumerate}

\subsection{Probability of Good Blocks}
\begin{enumerate}

\item[III.]
Most blocks are ``good".
\begin{equation}
\label{xgood}
\mathbb{P}(X\in G_j^\mathbb{X})\geq 1-L_j^{-\delta}.
\end{equation}
\begin{equation}
\label{ygood}
\mathbb{P}(Y\in G_j^\mathbb{Y})\geq 1-L_j^{-\delta}.
\end{equation}

%
\end{enumerate}

\subsection{Consequences of the Estimates}
For now let us assume that the estimates $I-III$ hold at some level $j$.  Then we have the following consequences (we only state the results for $\mathbb{X}$, but similar results hold for $\mathbb{Y}$ as well).

\begin{lemma}\label{l:consequences}
Let us suppose \eqref{tailx1} and \eqref{xgood} hold at some level $j$. Then for all $X\in G_j^{\mathbb{X}}$ we have the following.
\begin{enumerate}
\item[(i)]
\begin{equation}
\label{e:ccgg}
\P[X\cc Y\mid Y\in G_j^{\mathbb{Y}}, X]\geq \frac{3}{4}+2^{-(j+7/2)}.
\end{equation}

\item[(ii)]
\begin{equation}
\label{e:csgg}
\P[X\cs Y\mid Y\in G_j^{\mathbb{Y}}, X]\geq \frac{9}{10}+2^{-(j+7/2)}, \quad \P[X\scc Y\mid Y\in G_j^{Y}, X]\geq \frac{9}{10}+2^{-(j+7/2)}.
\end{equation}

\item[(iii)]
\begin{equation}
\label{e:ssgg}
\P[X\ssc Y\mid Y\in G_j^{\mathbb{Y}}, X]\geq 1-L_j^{-\beta}.
\end{equation}
\end{enumerate}
\end{lemma}

\begin{proof}
We only prove \eqref{e:ssgg}, other two are similar. We have
\[
\P[X\not\ssc Y\mid Y\in G_j^{\mathbb{Y}},X]\leq \dfrac{\P[X\not\ssc Y\mid X]}{\P[Y\in G_j^{\mathbb{Y}}]}\leq L_j^{-2\beta} (1-L_j^{-\delta})^{-1}\leq L_j^{-\beta}
\]
which implies \eqref{e:ssgg}.
\end{proof}

\begin{theorem}[Recursive Theorem]
\label{induction}
There exist positive constants $\alpha$, $\beta$, $\delta$, $m$, $k_0$ and $R$ such that for all large enough $L_0$ the following holds.
If the recursive estimates (\ref{tailx1}), (\ref{taily1}), (\ref{lengthx}), (\ref{lengthy}), (\ref{xgood}), (\ref{ygood})  and hold at level $j$ for some $j\geq 1$ then all the estimates hold at level $(j+1)$ as well.
\end{theorem}

We will choose the parameters as in equation~\eqref{e:parameters}.  Before giving a proof of Theorem \ref{induction} we show how using this theorem we can prove the general theorem. To use the recursive theorem we first need to show that the estimates $I$ and $II$ hold at the base level $j=1$.  Because of the obvious symmetry between $\mathbb{X}$ and $\mathbb{Y}$ we need only show that (\ref{tailx1}), (\ref{lengthx}) and (\ref{xgood}) hold for $j=1$ if $M$ is sufficiently large.

\subsection{Proving the Recursive Estimates at Level 1}
Let $X=(X_{1}^{(0)}, X_{(2)}^{(0)}, \ldots , X_{(L_1+
T_X^{(1)})}^{(0)})\sim \mu_{1}^{\mathbb{X}}$ be an $\mathbb{X}$-block at level $1$. Let $Y=(Y_{1}^{(0)}, Y_{(2)}^{(0)}, \ldots , Y_{(L_1+T_Y^{(1)})}^{(0)})\sim \mu_{1}^{\mathbb{Y}}$.

\begin{theorem}\label{t:basecase}
For all sufficiently large $L_0$, if $M$ (depending on $L_0$) is sufficiently large, then
\begin{equation}
\label{tailxbase}
\mathbb{P}(S_j^{\mathbb{X}}(X)\leq p)\leq p^{m+2^{-1}}L_1^{-\beta}~~\text{for}~~p\leq \frac 34 + 2^{-4},
\end{equation}
and
\begin{equation}
\label{xgoodbase}
\mathbb{P}(X\in G_j^\mathbb{X})\geq 1-L_1^{-\delta}.
\end{equation}
\end{theorem}

Theorem \ref{t:basecase} is proved using the following Lemmas. Without loss of generality we shall assume that $M$ is a multiple of $4$.

\begin{lemma}\label{l:basecaselength}
Let $X$ be an $\mathbb{X}$ block at level 1 as above. Then we have for all $l\geq 1$,
\begin{equation}
\label{e:baselengthbound}
\P(T_X^{(1)}\geq l)\leq \left(\frac{15}{16}\right)^{\frac{l-1}{2}}.
\end{equation}
Further we have,
\begin{equation}
\label{e:baselengthbound2}
\mathbb{E}[\exp (L_{0}^{-6}(|X|-\frac{3}{2}L_1))] \leq 1.
\end{equation}

\end{lemma}
\begin{proof}
It follows from the construction of blocks at level $1$ that
$T_X^{(1)} \preceq 2V$
where $V$ has a $\mbox{Geom}(1/16)$ distribution, (\ref{e:baselengthbound}) follows immediately from this.
To prove (\ref{e:baselengthbound2}) we notice the following two facts.

\[
\P[\exp(L_0^{-6}(|X|-3/2L_1))\geq \frac{1}{2}]\leq \P[|X|\geq \frac{3}{2}L_1-L_0^{6}\log 2]\leq \P[|X|\geq 5/4L_1\leq (15/16)^{\frac{L_1}{10}}\leq 1/4
\]
for $L_0$ large enough using (\ref{e:baselengthbound}).

Also, for all $x\geq 0$ using  (\ref{e:baselengthbound}),
\[
\P[\frac{|X|-3/2L_1}{L_0^6}\geq x]\leq \left(\frac{15}{16}\right)^{xL_0^6/2+L_1/4} \leq \frac{1}{10}\exp(-3x).
\]

Now it follows from above that

\begin{eqnarray*}
\E[\exp(L_0^{-6}(|X|-3/2L_1))] &=&\int_{0}^{\infty}\P[\exp(L_0^{-6}(|X|-3/2L_1))\geq y]~dy\\
&=& \int_{0}^{\frac12}\P[\exp(L_0^{-6}(|X|-3/2L_1))\geq y]~dy\\
&+&\int_{\frac12}^{1}\P[\exp(L_0^{-6}(|X|-3/2L_1))\geq y]~dy\\
&+& \int_{1}^{\infty}\P[\exp(L_0^{-6}(|X|-3/2L_1))\geq y]~dy\\
&\leq & \frac12 + \frac18  + \frac{1}{10}\int_{0}^{\infty}\P[(L_0^{-6}(|X|-3/2L_1))\geq z]e^z~dz\\
&\leq & \frac12 +\frac18 + \frac{1}{10}\leq 1.
\end{eqnarray*}
This completes the proof.
\end{proof}

We define $\A{1}_{X,1}$ to be the set of level $1$ $\mathbb{X}$-blocks defined by
\[
\A{1}_{X,1} := \left\{X:T_X^{(1)} \leq 100mL_1\right\}.
\]

It follows from Lemma \ref{l:basecaselength} that for $L_0$ sufficiently large
\begin{equation}
\label{e:A_1basesizebound}
\P(X\in \A{1}_{X,1})\geq 1-L_1^{-3\beta}.
\end{equation}

\begin{lemma}\label{l:basesideestimates}
For $M$ sufficiently large, the following inequalities hold for each $X\in \A{1}_{X,1}$.
\begin{enumerate}
\item[(i)]
\begin{equation}\label{e:basecornertocorner}
\P[X\cc Y\mid X]\geq \frac{3}{4}+2^{-4}.
\end{equation}

\item[(ii)]
\begin{equation}\label{e:basecornertoside}
\P[X\cs Y\mid X]\geq \frac{9}{10}+2^{-4}~\text{and}~\P[X\scc Y\mid X]\geq \frac{9}{10}+2^{-4}.
\end{equation}
\item[(iii)]
\begin{equation}\label{e:basesidetoside}
\P[X\ssc Y\mid X]\geq 1-L_1^{-2\beta}.
\end{equation}
\end{enumerate}
\end{lemma}

\begin{proof}
Let $Y$ be a level $1$ block constructed out of the sequence $Y_1^{(0)},\ldots$. Let $\mathcal{C}(X)$ be the event

\[
\left\{Y_i^{(0)}\neq X_{i'}^{(0)} \forall~i,i', i\in [(10m+1)L_1], i'\in [L_1+T_X^{(1)}]\right\}.
\]

Let $\mathcal{E}$ denote the event

\[
\left\{Y\in \A{1}_{Y,1}\right\}.
\]

Using the definition of the sequence $Y_1^{(0)},\ldots$ and the $\mathbb{Y}$-version of \eqref{e:A_1basesizebound} we get that

\[
\P[\mathcal{C}(X) \cap \mathcal{E}\mid X]\geq \left(1-\frac{4(100m+1)L_1}{M}\right)^{(100m+1)L_1}-L_1^{-3\beta}\geq \max\left\{1-L_1^{-2\beta},\frac{9}{10}+2^{-4}\right\}
\]
for $M$ large enough.

Since $X\ssc Y$, $X\scc Y$, $X\cs Y$, $X\cc Y$ each hold if $\mathcal{C}(X)$ and $\mathcal{E}$ both hold, the lemma follows immediately.
\end{proof}

\begin{lemma}\label{l:baseccestimates}
If $M$ is sufficiently large then
\begin{equation}
\mathbb{P}(\P(X\cc Y\mid X)\leq p)\leq p^{m+\frac12}L_1^{-\beta}~~\text{for}~~p\leq \frac 34 + 2^{-4}.
\end{equation}
\end{lemma}

\begin{proof}
Since $L_1$ is sufficiently large, \eqref{e:basecornertocorner} implies that it suffices to consider the case  $p< \frac{1}{500}$ and $X\notin \A{1}_{X,1}$. We prove that for $p<\frac{1}{500}$

\begin{equation}
\label{e:basecornertocornerB}
\P[\P(X\cc Y\mid X)\leq p, X\notin \A{1}_{X,1}]\leq p^{m+2^{-1}}L_1^{-\beta}.
\end{equation}

Let $\mathcal{E}(X)$ denote the event

\[\{T_Y^{(1)}=\lfloor\frac{1}{50m}T_X^{(1)}\rfloor, Y_{i}^{(0)}\neq 2~\mbox{mod}~4. \forall i\in [L_1+1,L_1+T_Y^{(1)}]\}\]

It follows from definition that

\begin{equation}
\label{e:basecaseebound}
\P[\mathcal{E}(X)\mid X]\geq \left(\frac{1}{4}\right)^2\left(\frac{3}{4}\right)^{\frac{T_X^{(1)}}{50m}}.
\end{equation}

Now let $D_k$ denote the event that
\[D_k=\{Y_k^{(0)}\neq X_{i'}^{(0)} \forall ~i'\in [ 50km , 50(k+2)m \wedge T_Y^{(1)} \}.
\]

Let
\[D=\bigcap_{k=1}^{L_1+T_Y^{(1)}} D_k
\]

It follows that

\[
\P[D_k\mid X,\mathcal{E}(X)]\geq (1-\frac{400m}{M}).
\]

Since $D_k$ are independent conditional on $X$ and $\mathcal{E}(X)$

\[
\P[D\mid X,\mathcal{E}(X)]\geq (1-400m/M)^{L_1+T_X^{(1)}/50m}.
\]

It follows that
\begin{eqnarray*}
\P[X\cc Y \mid X] &\geq & \left(\frac{1}{4}\right)^2\left(\frac{3}{4}\right)^{\frac{T_X^{(1)}}{50m}}\left(1-\frac{200m}{M}\right)^{L_1+T_X^{(1)}/50m}\\
&\geq & \frac{1}{20}\left(\frac{7}{10}\right)^{\frac{T_X^{(1)}}{50m}}
\end{eqnarray*}
for $M$ sufficiently large.

It follows that
\begin{eqnarray*}
\P[\P(X\cc Y\mid X) \leq  p, X\notin \A{1}_{X,1}] &\leq & \P[T_X^{(1)} \geq (50m\dfrac{\log 20p}{\log \frac{7}{10}})\vee 100mL_1]\\
&\leq & \left(\frac{15}{16}\right)^{20m\dfrac{\log 20p}{\log \frac{7}{10}}}\wedge \left(\frac{15}{16}\right)^{40mL_1}\\
&\leq & (20p)^{2m} \wedge \left(\frac{15}{16}\right)^{40mL_1}\leq p^{m+2^{-1}}L_1^{-\beta}
\end{eqnarray*}

since $(15/16)^{10} < 7/10$ and $L_0$ is sufficiently large and $m>100$.
\end{proof}


\begin{proof}[Proof of Theorem \ref{t:basecase}]
We have established (\ref{tailxbase}) in Lemma \ref{l:baseccestimates}. That (\ref{xgoodbase}) holds follows from Lemma \ref{l:basesideestimates} and \eqref{e:A_1basesizebound} noting $\beta >\delta$.
\end{proof}
\\

Now we prove Theorem~\ref{t:main} using Theorem~\ref{induction}.\\

\begin{proof}[of Theorem~\ref{t:main}]
Let $\mathbb{X}=(X_1,X_2,\ldots)$, $\mathbb{Y}=(Y_1,Y_2,\ldots)$ be as in the statement of the theorem. Let for $j\geq 1$, $\mathbb{X}=(X_1^{(j)}, X_2^{(j)},\ldots)$ denote the partition of $\mathbb{X}$ into level $j$ blocks as described above. Similarly let $\mathbb{Y}=(Y_1^{(j)}, Y_2^{(j)},\ldots)$ denote the partition of $\mathbb{Y}$ into level $j$ blocks. Let $\beta, \delta , m, R$ be as in Theorem~\ref{induction}. It follows form Theorem \ref{t:basecase} that for all sufficiently large $L_0$, estimates $I$ and $II$ hold for $j=1$ for all sufficiently large $M$. Hence the  Theorem \ref{induction} implies that if $L_0$ is sufficiently large then  $I$ and $II$  hold for all $j\geq 1$ for $M$ sufficiently large.

Let $\mathcal{T}_j^{\mathbb{X}}=\{X_k^{(j)}\in G_{j}^{\mathbb{X}}, 1\leq k \leq L_j^3\}$ be the event that the first $L_j^3$ blocks at level $j$ are good.  Notice that on the event $\cap_{k=1}^{j-1} \mathcal{T}_k^{\mathbb{X}}$,  $X_1^{(j)}$ has distribution $\mu_j^{\mathbb{X}}$ by Observation~\ref{o:blockStructure} and so $\{X_i^{(j)}\}_{i\geq 1}$ is i.i.d. with distribution $\mu_j^{\mathbb{X}}$. Hence it follows from equation (\ref{xgood}) that $\mathbb{P}(\mathcal{T}_j^{\mathbb{X}}|\cap_{k=1}^{j-1} \mathcal{T}_k^{\mathbb{X}})\geq (1-L_j^{-\delta})^{L_j^3}$. Similarly defining $\mathcal{T}_j^{\mathbb{Y}}=\{Y_k^{(j)}\in G_{j}^{\mathbb{Y}}, 1\leq k \leq L_j^3\}$ we get using (\ref{ygood}) that $\mathbb{P}(\mathcal{T}_j^{\mathbb{Y}}|\cap_{k=0}^{j-1} \mathcal{T}_k^{\mathbb{Y}})\geq (1-L_j^{-\delta})^{L_j^3}$.

Let $\mathcal{A}=\cap _{j\geq 0}(\mathcal{T}_j^{\mathbb{X}}\cap \mathcal{T}_j^{\mathbb{Y}})$. It follows from above that $\mathbb{P}(\mathcal{A})>0$ since $\delta>3$

Let $\mathcal{A}_{j+1}=\cap _{k\leq j}(\mathcal{T}_k^{\mathbb{X}}\cap \mathcal{T}_k^{\mathbb{Y}})$. It follows from \eqref{e:ccgg} and \eqref{xgood} that
\[
P[X_1^{(j+1)}\cc Y_1^{(j+1)}\mid \mathcal{A}_{j+1}]\geq \frac{3}{4}+2^{-(j+9/2)}-2L_{j+1}^{-\delta}\geq \frac{3}{4}.
\]

Let $\mathcal{B}_{j+1}$ denote the event

\[
\mathcal{B}_{j+1}=\left\{\exists ~\text{an open path from}~(0,0)\rightarrow (m,n)~\text{for some}~m,n\geq L_{j+1}  \right\}.
\]

Then $\mathcal{B}_{j+1} \downarrow$ and $\mathcal{B}_{j+1}\supseteq \{X_1^{(j+1)}\cc Y_1^{(j+1)}\}$. It follows that

\[\P[\cap \mathcal{B}_{j+1}]\geq \liminf P[X_1^{(j+1)}\cc Y_1^{(j+1)}]\geq \frac{3}{4}\P[\mathcal{A}]>0.
\]

A standard compactness argument shows that $\cap \mathcal{B}_{j+1}\subseteq \{\mathbb{X}\leftrightarrow \mathbb{Y}\}$ and hence $\P[\mathbb{X}\leftrightarrow \mathbb{Y}]>0$, which completes the proof of the theorem.
\end{proof}

The remainder of the paper is devoted to the proof of the estimates in the induction. Throughout these sections we assume that the estimates $I-III$ hold for some level $j\geq 1$ and then prove the estimates at level $j+1$.  Combined they complete the proof of Theorem~\ref{induction}.

From now on, in every Theorem, Proposition and Lemma we state, we would
implicitly assume the hypothesis that all the recursive estimates hold upto level $j$, the parameters satisfy the constraints described in \S~\ref{s:parameters} and $L_0$ is sufficienctly large.

\section{Geometric Constructions}

We shall join paths across blocks at a lower level two form paths across blocks at a higher level. The general strategy will be as follows. Suppose we want to construct a path across $X\times Y$ where $X$, $Y$ are level $j+1$ blocks. Using the recursive estimates at level $j$ we know we are likely to find many paths across $X_i\times Y$ where $X_i$ is a good sub-block of $X$. So we need to take special care to ensure that we can find open paths  crossing bad-subblocks of $X$ (or $Y$). To show the existence of such paths, we need some geometric constructions, which we shall describe in this section. We start with the following definition.

\begin{definition}[Admissible Assignments]
\label{d:assign}
Let $I_1=[a+1,a+t]\cap \Z$ and $I_2=[b+1,b+t']\cap \Z$ be two intervals of consecutive positive integers. Let $I_1^{*}=[a+L_j^3+1, a+t-L_j^3]\cap \Z$ and $I_2^{*}=[b+L_j^3+1, b+t'-L_j^3]\cap \Z$. Also let $B\subseteq I_1^{*}$ and $B'\subseteq I_2^{*}$ be given. We call $\Upsilon(I_1,I_2,B,B')=(H,H',\tau)$ to be an admissible assignment at level $j$ of $(I_1,I_2)$ w.r.t. $(B,B')$ if the following conditions hold.

\begin{enumerate}
\item[(i)] $B\subseteq H=\{a_1<a_2<\cdots <a_{\ell}\}\subseteq I_1$ and $B'\subseteq H'=\{b_1<b_2<\cdots <b_{\ell}\}\subseteq I_2^{*}$ with $\ell=|B|+|B'|$.
\item[(ii)] $\tau(a_i)=b_i$ and $\tau(B)\cap B'=\emptyset$.
\item[(iii)] Set $a_0=a, a_{\ell+1}=a+t+1$; $b_0=b,b_{\ell+1}=b+t'+1$. Then we have for all $i\geq 0$
$$\frac{1-2^{-(j+7/2)}}{R}\leq \frac{b_{i+1}-b_i-1}{a_{i+1}-a_i-1}\leq R(1+2^{-(j+7/2)}).$$
\end{enumerate}
\end{definition}
The following proposition concerning the existence of admissible assignment follows from the results in Section 6 of \cite{BS14}. We omit the proof.

\begin{proposition}
\label{p:assign}
Assume the set-up in Definition \ref{d:assign}. We have the following.
\begin{enumerate}
\item[(i)] Suppose we have
$$\frac{1-2^{-(j+4)}}{R}\leq \frac{t'}{t}\leq R(1+2^{-(j+4)}).$$ Also suppose $|B|,|B'|\leq 3k_0$. Then there exist $L_j^2$ level $j$ admissible assignments $(H_i,H'_i,\tau_i)$ of $(I_1,I_2)$ w.r.t. $(B,B')$ such that for all $x\in B$, $\tau_{i}(x)=\tau_{1}(x)+i-1$ and for all $y\in B'$, $\tau_i^{-1}(y)=\tau_1^{-1}(y)-i+1$.

\item[(ii)] Suppose
$$\frac{3}{2R}\leq \frac{t'}{t}\leq \frac{2R}{3}$$
and $|B|\leq \frac{t-2L_j^3}{10R_j^{+}}$. Then there exists an admissible assignment $(H,H',\tau)$ at level $j$ of $(I_1,I_2)$ w.r.t. $(B,\emptyset)$.
\end{enumerate}
\end{proposition}

Constructing suitable admissible assignments will let us construct different types of open paths in different rectangles. To demonstrate this we first define the following somewhat abstract set-up.

\subsection{Admissible Connections}
Assume the set-up in Definition \ref{d:assign}. Consider the lattice $A=I_1\times I_2$. Let $\mathcal{B}=\left( B_{i_1,i_2}\right)_{(i_1,i_2)\in A}$ be a collection of finite rectangles where $B_{i_1,i_2}=[n_{i_1}]\times [n'_{i_2}]$. Let $A\otimes \mathcal{B}$ denote the bi-indexed collection
$$\left\{\left((a_1,b_1),(a_2,b_2)\right): (a_1,a_2)\in A, (b_1,b_2)\in B_{a_1,a_2} \right\}.$$

We think of $A\otimes \mathcal{B}$ as a $\sum_{i_1}n_{i_1}\times \sum_{i_2}n'_{i_2}$ rectangle which is further divided into rectangles indexed by $(i_1,i_2)\in A$ in the obvious manner.

\begin{definition}[Route]
\label{d:route}
A {\bf route} $P$ at level $j$ in $A\otimes \mathcal{B}$ is a sequence of points $((v_i,b^{1,v_i}),(v_i,b^{2,v_i}))_{i\in [\ell]}$ in $A\otimes \mathcal{B}$ satisfying the following conditions.
\begin{enumerate}
\item[(i)] $V(P)=\{v_1,v_2,\ldots, v_{\ell}\}$ is an oriented path from $(a+1,b+1)$ to $(a+t,b+t')$ in $A$.
\item[(ii)] Let $v_i=(v_i^1,v_i^2)$. For each $i$, $b^{1,v_i}\in [L_{j-1},n_{v_i^{1}}-L_{j-1}]\times \{1\} \cup \{1\}\times [L_{j-1},n'_{v_i^{2}}-L_{j-1}]$ and $b^{2,v_i}\in [L_{j-1},n_{v_i^{1}}-L_{j-1}]\times \{n'_{v_i^{2}}\} \cup \{n_{v_i^{1}}\}\times [L_{j-1},n'_{v_i^{2}}-L_{j-1}]$ except that $b^{1,v_1}=(1,1)$ and $b^{2,v_{\ell}}=(n_{v_{\ell}^{1}},n'_{v_{ell}^{2}})$ are also allowed.
\item[(iii)] For each $i$ (we drop the superscript $v_i$), let $b^{1}=(b^1_1,b^1_2)$ and $b^2=(b^2_1,b^2_2)$. Then for each $i$, we have $\frac{1-2^{-(j+3)}}{R} \leq \frac{b^2_2-b_2^1}{b_1^2-b_1^1}\leq R(1+2^{-(j+3)}).$
\item[(iv)] $b^{2,v_i}$ and $b^{1,v_{i+1}}$ agree in one co-ordinate.
\end{enumerate}

A route $P$ defined as above is called a route in $A\otimes B$ from $(v_1,b^{1,v_1})$ to $(v_{\ell},b^{2,v_{\ell}})$. We call $P$ a {\bf corner to corner route} if $b^{1,v_1}=(1,1)$ and $b^{2,v_{\ell}}=(n_{v_{\ell}^{1}},n'_{v_{\ell}^{2}})$.
For $k\in I_2$, the $k$-section of the route $P$ is defined to be the set of $k'\in I_1$ such that $(k',k)\in V(P)$.
\end{definition}

Now gluing together these routes  one can construct corner to corner (resp. corner to side or side to side) paths under certain circumstances. We make the following definition to that end.

\begin{definition}[Admissible Connections]
Consider the above set-up. Let $S_{in}=[L_{j-1},n_{a+1}-L_{j-1}]\times \{1\} \cup \{1\}\times [L_{j-1},n'_{b+1}-L_{j-1}]$ and $S_{out}=[L_{j-1},n_{a+t}-L_{j-1}]\times \{n'_{b+t'}\} \cup \{n_{a+t}\}\times [L_{j-1},n'_{b+t}-L_{j-1}]$. Suppose for each $b\in S_{out}$ there exists a level $j$ route $P^b$ in $A\otimes \mathcal{B}$ from $(1,1)$ to $b$. The collection $\mathcal{P}=\{P^b\}$ is called a corner to side admissible connection in $A\otimes \mathcal{B}$. A side to corner admissible connection is defined in a similar manner. Now suppose for each $b\in S_{in}$, $b'\in S_{out}$ there exists a level $j$ route $P^{b,b'}$ in $A\otimes \mathcal{B}$ from $b$ to $b'$. The collection $\mathcal{P}=\{P^{b,b'}\}$ in this case is called a side to side admissible connection in $A\otimes \mathcal{B}$. We also define $V(\mathcal{P})=\cup_{P\in \mathcal{P}} V(P)$.
\end{definition}

The usefulness of having these abstract definitions is demonstrated by the next few lemmata. These follow directly from definition and hence we shall omit the proofs.

Now let $X=(X_1,X_2,....,X_t)$ be an $\mathbb{X}$-blocks at level $j+1$ with $X_i$ being the $j$-level subblicks constituting it. Let $X_i$ consisting of $n_i$ many chunks of $(j-1)$-level subblocks. Similarly let $Y=(Y_1,Y_2,...,Y_{t'})$ be a $\mathbb{Y}$-block at level $j+1$ with $j$-level subblocks $Y_i$ consisting of $n'_i$ many chunks of $(j-1)$ level subblocks. Then we have the following Lemmata. Set $A=[t]\times [t']$. Define $\mathcal{B}=\{B_{i,j}\}$ where $B_{i_1,i_2}=[n'_{i_1}]\times [n'_{i_2}]$.

\begin{figure}[h]
\begin{center}
\includegraphics[height=5cm, width=5cm]{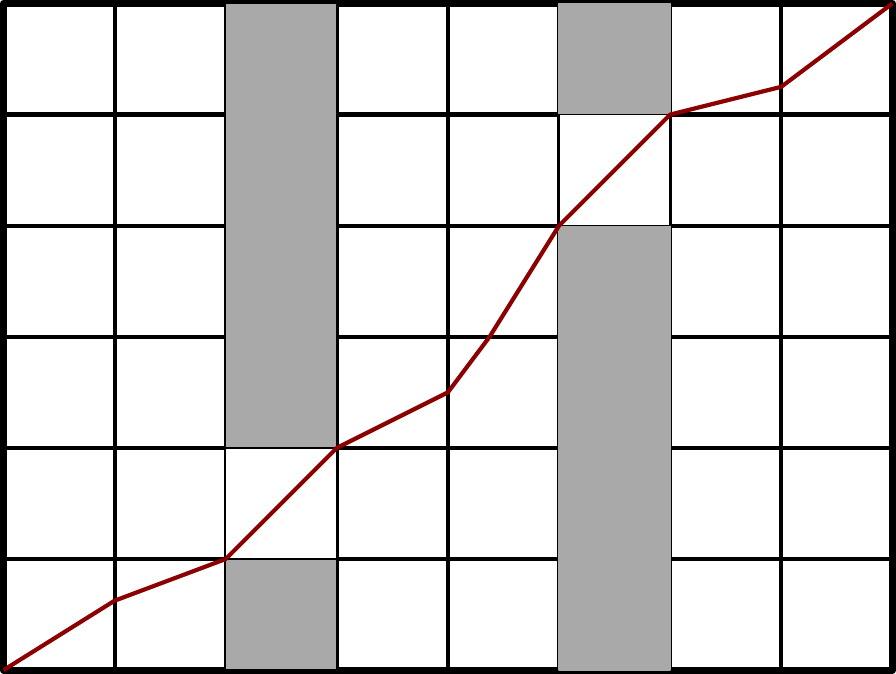}
\end{center}
\caption{Corner to Corner Paths}
\end{figure}

\begin{lemma}
\label{l:ccconnect}
Consider the set-up described above. Let $H=\{a_1<a_2<\cdots <a_{\ell}\}\subseteq [t]$ and $H'=\{b_1<b_2<\cdots <b_{\ell}\}$. Set $\tilde{X}_{(s)}=(X_{a_s+1},\ldots, X_{a_{s+1}-1})$ and $\tilde{Y}_{(s)}=(Y_{b_s+1},\ldots, Y_{b_{s+1}-1})$. Suppose further that for each $s$, $\tilde{X}_{(s)}\cc \tilde{Y_{(s)}}$ and $X_{a_s}\cc Y_{b_s}$. Then we have $X\cc Y$.
\end{lemma}

The next lemma gives sufficient conditions under which we have $\tilde{X}_{(s)}\cc \tilde{Y}_{(s)}$.

\begin{lemma}
\label{l:ccconnect2}
In the above set-up, let $I_1^s=[a_{s}+1,a_{s+1}-1]$, $I_2^s=[a_{s}+1,a_{s+1}-1]$. Set $A^s=I_1^s\times I_2^s$ and let $\mathcal{B}^s$ be the restriction of $\mathcal{B}$ to $A^s$. Suppose there exists a corner to corner route $P$ in $A^s\otimes B^s$ such that $X_{a_s+1}\cs Y_{b_s+1}$, $X_{a_{s+1}-1}\scc Y_{b_{s+1}-1}$ and for all other $(v_1,v_2)\in V(P)$ $X_{v_1}\ssc Y_{v_2}$. Then $\tilde{X}_{(s)}\cc \tilde{Y}_{(s)}$.
\end{lemma}

The above lemmata are immediate from definition. Now we turn to corner to side, side to corner and side to side connections. We have the following lemma.

\begin{lemma}
\label{l:csssconnect}
Consider the set-up as above. Suppose $X$ and $Y$ contain $n_X$ and $n_Y$ many chunks respectively. Further suppose that none of the subblock $X_i$ or $Y_i$ contain more than $3L_j$ level $0$ subblocks.
\begin{enumerate}
\item[(i)] Suppose for every exit chunk in $\mathcal{E}_{out}(X,Y)$ the following holds.
For concreteness consider the chunk $(k,n_{Y})$. Let $T_{k}$ denote the set of all $i$ such that $X_{i}$ is contained in $C_k^{X}$. There exists $T_{k}^{*}\subseteq T_{k}$ with $|T_k^{*}|\geq (1-10k_0L_j^{-1})|T_k|$ such that for all $r\in T_{k}^{*}$ and $\tilde{X}=(X_1,\ldots, X_r)$ we have $\tilde{X}\cso Y$.

Then we have $X\cs Y$.

\item[(ii)] A similar statement holds for $X\scc Y$.

\item[(iii)] Suppose for every pair of entry-exit chunks in $\mathcal{E}(X,Y)$ the following holds.
For concreteness consider the pair of entry-exit chunks $((k_1,1), (n_X,k_2))$. Let $T_{k_1}$ (resp. $T'_{k_2}$) denote the set of all $i$ such that $X_{i}$ (resp. $Y_i$) is contained in $C_{k_1}^{X}$ (resp. $C_{k_2}^{Y}$). There exists $T_{k_1,*}\subseteq T_{k_1}$, $T'_{k_2,*}\subseteq T'_{k_2}$  with $|T_{k_1,*}|\geq (1-10k_0L_j^{-1})|T_{k_1}|$,  $|T'_{k_2,*}|\geq (1-10k_0L_j^{-1})|T'_{k_2}|$ such that for all $r\in T_{k_1,*}$, $r'\in T'_{k_2,*}$ and $\tilde{X}=(X_r,\ldots, X_t)$, $\tilde{Y}=(Y_1,\ldots, Y_{r'})$  we have $\tilde{X}\ssco Y$.

Then we have $X\ssc Y$.
\end{enumerate}
\end{lemma}

\begin{proof}
Parts $(i)$ and $(ii)$ are straightforward from definitions. Part $(iii)$ follows from definitions by noting the following consequence of planarity. Suppose there are open oriented paths in $\Z^2$ from $v_1=(x_1,y_1)$ to $v_2=(x_2,y_2)$ and also from $v_3=(x_3,y_1)$ to $v_4(x_2,y_3)$ such that $x_1<x_3<x_2$ and $y_1<y_2<y_3$. Then these paths must intersect and hence there are open paths from $v_1$ to $v_4$ and also from $v_2$ to $v_3$. The condition on the length of sub-blocks is used to ensure that none of the subblocks in $T_{k_1}\setminus T_{k_1,*}$ are extremely long.
\end{proof}

\begin{figure}[h]
\begin{center}
\includegraphics[height=5cm, width=11cm]{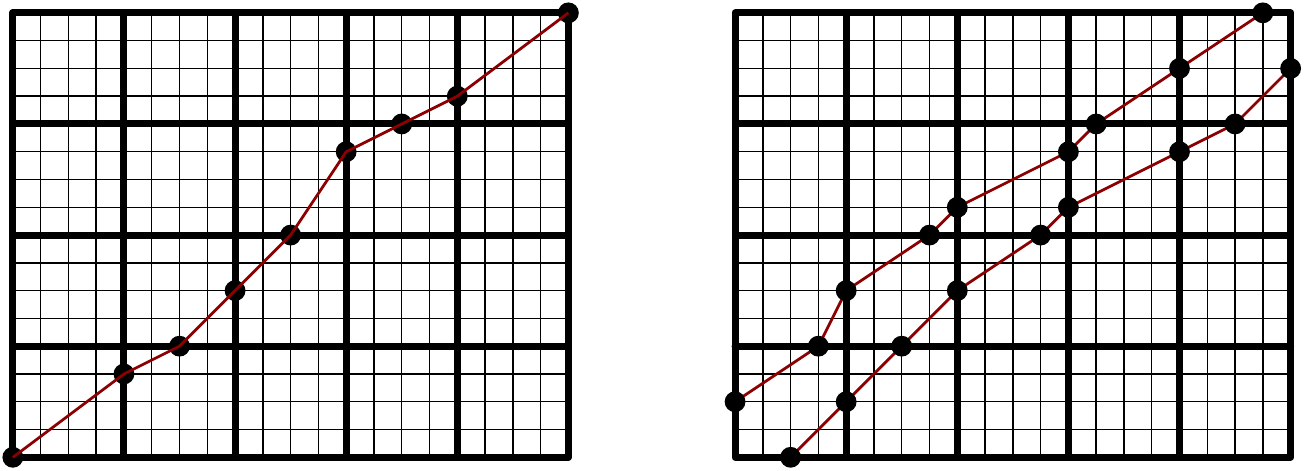}
\end{center}
\caption{Corner to Corner and Side to Side routes}
\end{figure}

The next lemma gives sufficient conditions for $\tilde{X}\cso \tilde{Y}$ and $\tilde{X}\ssco \tilde{Y}$ in the set-up of the above lemma. This lemma also easily follows from definitions.

\begin{lemma}
\label{l:csscconnect}
Assume the set-up of Lemma \ref{l:csssconnect}. Let $\tilde{X}=(X_{t_1},X_{t_1+1},\ldots ,X_{t_2})$ and $\tilde{Y}=(Y_{t'_1},\ldots , Y_{t'_2})$. Let $H=\{a_1<a_2<\cdots <a_{\ell}\}\subseteq [t_1,t_2]$ and $H'=\{b_1<b_2<\cdots <b_{\ell}\}\subseteq [t'_1,t'_2]$. Set $\tilde{X}_{(s)}=(X_{a_s+1},\ldots, X_{a_{s+1}-1})$ and $\tilde{Y}_{(s)}=(Y_{b_s+1},\ldots, Y_{b_{s+1}-1})$ ($a_0,b_0$ etc. are defined in the natural way).
\begin{enumerate}
\item[(i)] Suppose  that for each $s<\ell$, $\tilde{X}_{(s)}\cc \tilde{Y_{(s)}}$ and $\tilde{X}_{(\ell)}\cso \tilde{Y_{(\ell)}}$. Also suppose for each $s$,  $X_{a_s}\cc Y_{b_s}$. Then we have $\tilde{X}\cso \tilde{Y}$.
\item[(ii)] A similar statement holds for $\tilde{X}\scco \tilde{Y}$.
\item[(ii)] Suppose  that for each $s\in [\ell -1]$, $\tilde{X}_{(s)}\cc \tilde{Y_{(s)}}$, $\tilde{X}_{(0)}\scco \tilde{Y_{(0)}}$  $\tilde{X}_{(\ell)}\cso \tilde{Y_{(\ell)}}$. Also suppose for each $s$,  $X_{a_s}\cc Y_{b_s}$. Then we have $\tilde{X}\ssco \tilde{Y}$.
\end{enumerate}
\end{lemma}

Now we give sufficient conditions for $\tilde{X}\cso \tilde{Y}$ and $\tilde{X}\scco \tilde{Y}$ in terms of routes.

\begin{lemma}
\label{l:cssssuccess}
In the above set-up, further suppose that none of the level $(j-1)$ sub-blocks of $X_{t_1}$, $X_{t_2}$, $Y_{t'_1}$, $Y_{t'_2}$ contain more than $3L_{j-1}$ level $0$ sub-blocks. Set $I_1^s=[a_{s}+1,a_{s+1}-1]$, $I_2^s=[a_{s}+1,a_{s+1}-1]$. Set $A^s=I_1^s\times I_2^s$ and let $\mathcal{B}^s$ be the restriction of $\mathcal{B}$ to $A^s$. Suppose there exists a corner to side admissible connection $\mathcal{P}$ in $A^s\otimes B^s$ such that $X_{a_s+1}\cs Y_{b_s+1}$ and for all other $(v_1,v_2)\in V(\mathcal{P})$ $X_{v_1}\ssc Y_{v_2}$. Then $\tilde{X}_{(s)}\cso \tilde{Y}_{(s)}$. Similar statements hold for $\tilde{X}_{(s)}\scco \tilde{Y}_{(s)}$ and $\tilde{X}_{(s)}\ssco \tilde{Y}_{(s)}$.
\end{lemma}

\begin{proof} Proof is immediate from definition of admissible connections and the inductive hypotheses (this is where we need the assumption on the lengths of $j-1$ level subblocks). For $\tilde{X}_{(s)}\ssco \tilde{Y}_{(s)}$, we again need to use planarity as before.
\end{proof}

Now we connect it up with the notion of admissible assignments defined earlier in this section. Consider the set-up in Lemma \ref{l:ccconnect}. Let $B_1\subseteq I_1=[t]$, $B_2\subseteq I_2=[t']$, let $B_1^*\supseteq B_1$ (resp. $B_2^*\supseteq B_2$) be the set containing elements of $B_1$ (resp. $B_2$) and its neighbours. Let $\Upsilon$ be a level $j$ admissible assignment of $(I_1,I_2)$ w.r.t. $(B_1^*,B_2^*)$ with associated $\tau$. Suppose $H=\tau^{-1}(B_2)\cup B_1$ and $H'=B_2^*\cup \tau(B_2)$. We have the following lemmata.

\begin{lemma}
\label{l:ccsuccessstar}
Consider $(\tilde{X}_{(s)},\tilde{Y}_{(s)})$ in the above set-up. There exists a corner to corner route $P$ in $A^s\otimes \mathcal{B}^s$. Further for each $k\in I_2^s$, there exist sets $H_k^{\tau}\subseteq I_1^s$ with $|H_k^{\tau}|\leq L_j$ such that the $k$-section of the route $P$ is contained in $H_k^{\tau}$ for all $k$. In the special case where $t=t'$ and $\tau(i)=i$ for all $i$, one cas take $H_{k}^{\tau}=\{k-1,k,k+1\}$. Further Let $A'\subseteq A^s$ with $|A'|\leq k_0$. Suppose Futher that for all $v=(v_1,v_2)\in A'$ and for $i\in \{s,s+1\}$ we have $||v-(a_i,b_i)||_{\infty} \geq k_0R^310^{j+8}$. Then we can take $V(P)\cap A'=\emptyset$.
\end{lemma}

\begin{proof}
This lemma is a consequence of Lemma \ref{l:admissibletoy} below.
\end{proof}

\begin{lemma}
\label{l:ccsssstar}
In the above set-up, consider $(\tilde{X}_{(s)}, \tilde{Y}_{(s)})$. Assume for each $i\in [a_{s}+1, a_{s+1}-1]$, $i'\in [b_{s}+1, b_{s+1}-1]$ we have $L_{j-1}^{\alpha-5}\leq n_{i}, n'_{i'}\leq L_{j-1}^{\alpha-5}+L_{j-1}$.
Let $A'\subseteq A^s$ with $|A'|\leq k_0$. Suppose further that for all $v=(v_1,v_2)\in A'$ and for  $i\in {s,s+1}$ we have $||v-(a_i,b_i)||_{\infty} \geq k_0R^310^{j+8}$. Assume also $a_{s+1}-a_{s}, b_{s+1}-b_s \geq 5^{j+6}R$. Then there exists a corner to side (resp. side to corner, side to side) admissible connection $\mathcal{P}$ in $A^{s}\otimes \mathcal{B}^{s}$ such that $V(P)\cap A'=\emptyset$.
\end{lemma}

\begin{proof}
This lemma also follows from Lemma \ref{l:admissibletoy} below.
\end{proof}

\begin{lemma}
\label{l:admissibletoy}
Let $A\otimes \mathcal{B}$ be as in Definition \ref{d:route}. Assume that
$\frac{1-2^{-(j+7/2)}}{R}\leq \frac{t'}{t} \leq R(1+2^{-(j+7/2)})$, and $ L_{j-1}^{\alpha -5}+L_{j-1} \geq n_i,n'_{i'} \geq L_{j-1}^{\alpha -5}$. Then the following holds.

\begin{enumerate}
\item[(i)]
There exists a corner to corner route $P$ in $A\otimes \mathcal{B}$ where
$V(P)\subseteq R(A)$ where
$$R(A)=\{v=(v_1,v_2)\in A: |v-(a+xt, b+xt')|_1\leq 50~\text{for some}~x\in[0,1]\}.$$

\item[(ii)]
Further, if $t,t'\geq 5^{j+6}R$, then there exists a corner to side (resp. side to corner, side to side) admissible connection $\mathcal{P}$ with $V(\mathcal{P})\subseteq R(A)$.

\item[(iii)]
Let $A'$ be a given subset of $A$ with $|A'|\leq k_0$ such that $A'\bigcap ([k_0R^310^{j+8}]\times [k_0R^310^{j+8}] \cup ([n-k_0R^310^{j+8},n]\times[n'-k_0R^310^{j+8}, n'])=\emptyset$. Then there is a corner to corner route $P$ in $A\otimes \mathcal{B}$ such that $V(P)\cap A'=\emptyset$. Further, if $t,t'\geq 5^{j+6}R$, then there exists a corner to side (resp. side to corner, side to side) admissible connection $\mathcal{P}$ with $V(\mathcal{P})\cap A'=\emptyset$.
\end{enumerate}
\end{lemma}

\begin{figure}[h]
\begin{center}
\includegraphics[height=6cm, width=6cm]{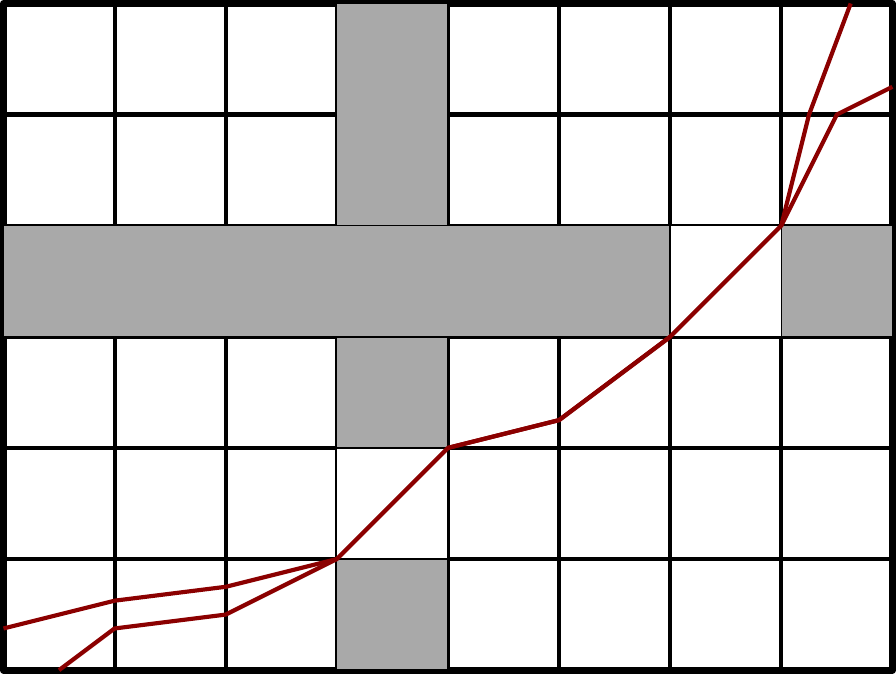}
\end{center}
\caption{Side to Side admissible connections}
\end{figure}

\begin{proof}
Without loss of generality, for this proof we shall assume $a=b=0$. We prove $(i)$ first. Let $y_i=\lfloor it'/t \rfloor +1$ for $i\in [t]$ and let $x_{i}= \lceil it/t'\rceil $ for $i\in [t']$. Define $\tilde{y_i}=(it'/t -y_i+1)$ and $\tilde{x_i}=(it/t'-x_i+1)$.

Define $y^{*}_{i}= \lfloor \tilde{y_i}n'_{y_i}\rfloor +1 $ and $x^{*}_{i}= \lceil \tilde{x_i}n_{x_i} \rceil$. Observe that it follows from the definitions that $y^{*}_{i}\in [n'_{y_i}]$ and $x^{*}_i \in [n_{x_i}]$. Now define $y^{**}_{i}=y^{*}_i$ if $y^{*}_i\in [L_{j-1},n'_{y_i}-L_{j-1}]$. If $y^{*}_i \in [L_{j-1}]$ define $y^{**}_i=L_{j-1}$, if $y^{*}_i \in [n'_{y_i}-L_{j-1},n'_{y_i}]$ define $y^{**}_i=n'_{y_i}-L_{j-1}$. Similarly define $x^{**}_{i}=x^{*}_i$ if $x^{*}_i\in [L_{j-1},n_{x_i}-L_{j-1}]$. If $x^{*}_i\in [L_{j-1}]$ define $x^{**}_i=L_{j-1}$, if $x^{*}_i \in [n_{x_i}-L_{j-1}, n_{x_i}]$ define $x^{**}_i=n_{x_i}-L_{j-1}$. Now for $i\in [t-1], i'\in [t'-1]$ consider points $((i,n_i), (y_i,y^{**}_i))$, $((i+1,1), (y_i,y^{**}_i))$, $((x_{i'},x^{**}_{i'}),(i', n'_{i'}))$, $((x_{i'},x^{**}_{i'}),(i'+1,1))$ alongwith the two corner points. We construct a corner to corner route using these points.

Let us define $V(P)=\{(i,y_i), (x_{i'},i'): i\in [t-1], i'\in [t'-1]\} \cup \{(t,t')\}$. We notice that either $y_1=1$ or $x_1=1$. It is easy to see that the vertices in $V(P)$ defines an oriented path from $(1,1)$ to $(t,t')$ in $A$. Denote the path by $(v^1,v^2,\ldots v^{t+t'-1})$. For $v=v^r, r\in [2,t+t'-2]$, we define points $b^{1,v^r}$ and $b^{2,v^r}$ as follows. Without loss of generality assume $v=v^r=(i,y_i)$. Then either $v^{r-1}=(i-1, y_{i})=(i-1, y_{i-1})$  or $v^{r-1}=(i,y_i-1)=(x_{y_i-1}, y_{i}-1)$. If $v^{r-1}=(i-1,y_{i-1})$, then define $\{(b_1^{1,v},b_2^{1,v}), (b_1^{2,v},b_2^{2,v})\}$ by $b_1^{1,v}=1$, $b_2^{1,v}=y^{**}_{i-1}$, $b_1^{2,v}=n_i$, $b_2^{2,v}=y^{**}_i$. If $v^{t-1}=(x_{y_i-1}, y_{i}-1)$ then define $K^v=\{(b_1^{1,v},b_2^{1,v}), (b_1^{2,v},b_2^{2,v})\}$ by $b_1^{1,v}=x^{**}_{y_i-1}$, $b_2^{1,v}=1$, $b_1^{2,v}=n_i$, $b_2^{2,v}=y^{**}_i$. To prove that this is indeed a route we only need to check the slope condition in Definition \ref{d:route} in both the cases. We do that only for the latter case and the former one can be treated similarly.

Notice that from the definition it follows that the slope between the points (in $\R ^2$) $(\tilde{x}_{y_i-1},0)$ and $(1,\tilde{y_{i}})$ is $\frac{t'}{t}$. We need to show that
$$\frac{1-2^{-(j+3)}}{R}\leq \frac{b_2^2-b_2^1}{b_1^2-b_1^1}= \frac{y^{**}_i-1}{n_i-x^{**}_{y_{i}-1}} \leq R(1+2^{-(j+3)})$$
where once more we have dropped the superscript $v$ for convenience. Now if $x^{*}_{y_i-1}\in [n_i-L_{j-1},n_i]$ and $y^{*}_{i}\in [L_{j-1}]$ then from definition it follows that $\frac{b_2^2-b_2^1}{b_1^2-b_1^1}=1$ and hence the slope condition holds. Next let us suppose $y^{*}_{i}\in [L_{j-1}]$ but
$x^{*}_{y_i-1}\notin [n_i-L_{j-1},n_i]$. Then clearly, $\frac{y^{**}_i-1}{n_i-x^{**}_{y_{i}-1}}\leq 1$. Also notice that in this case $x^{*}_{y_i-1}> L_{j-1}$ and $y_i^{**}-1\geq n'_{y_i}\tilde{y_i}(1-L_{j-1}^{-1})$. It follows that

$$1- \frac{x^{**}_{y_i-1}}{n_i}= 1- \frac{x^{*}_{y_i-1}}{n_i}\leq  1-\tilde{x}_{y_i-1}+\frac{1}{n_i} \leq (1-\tilde{x}_{y_i-1})(1+L_{j-1}^{-1}) .$$
Hence
$$ \frac{y^{**}_i-1}{n_i-x^{**}_{y_{i}-1}}\geq \frac{n'_{y_i}}{n_i}\frac{\tilde{y_i}}{1-\tilde{x}_{y_i-1}}\frac{1-L_{j-1}^{-1}}{1+L_{j-1}^{-1}}\geq \frac{t'}{t}\frac{L_{j-1}^{\alpha-5}(1-L_{j-1}^{-1})}{(L_{j-1}^{\alpha-5}+L_{j-1})(1+L_{j-1}^{-1})}\geq \frac{1-2^{-(j+3)}}{R}$$
for $L_0$ sufficiently large. The case where $y^{*}_{i}\notin [L_{j-1}]$ but
$x^{*}_{y_i-1}\in [n_i-L_{j-1},n_i]$ can be treated similarly.

Next we treat the case where $x^{*}_{y_i-1}\in [L_{j-1}+1, n_i-L_{j-1}-1]$ and $y^{*}_i\in [L_{j-1}+1, n'_{y_i}-L_{j-1}-1]$. Here we have similarly as before
$$(1-L_{j-1}^{-1})(1-\tilde{x}_{y_i-1}) \leq 1- \frac{x^{**}_{y_i-1}}{n_i} \leq (1-\tilde{x}_{y_i-1})(1+L_{j-1}^{-1})$$
and
$$\tilde{y_i}(1+2L_{j-1}^{-1})\geq  \frac{y_i^{**}-1}{n'_{y_i}}\geq \tilde{y_i}(1-2L_{j-1}^{-1}).$$

It follows as before that
\begin{eqnarray*}
\frac{(1+2L{j-1}^{-1})}{1-L_{j-1}^{-1}}\frac{n'_{y_i}}{n_i}\frac{n'}{n} \geq \frac{y^{**}_i-1}{n_i-x^{**}_{y_{i}-1}}\geq \frac{n'_{y_i}}{n_i}\frac{t'}{t}\frac{(1+2L{j-1}^{-1})}{1-L_{j-1}^{-1}}
\end{eqnarray*}
and hence
$$R(1+2^{-(j+3)}) \frac{y^{**}_i-1}{n_i-x^{**}_{y_{i}-1}}\geq \frac{1-2^{-(j+3)}}{R}$$
for $L_0$ sufficiently large.

Other cases can be treated in similar vein and we only provide details in the case where $y^{*}_i\in [n'_{y_i}-L_{j-1},n'_{y_i}]$ and $x^{*}_{y_i-1}\in [L_{j-1}]$. In this case we have that $$\tilde{y_i}(1-\frac{2L_{j-1}}{n'_{y_i}})\leq \frac{y^{**}_i-1}{n'_{y_i}}\leq \tilde{y}_i. $$

We also have that
$$(1-\tilde{x}_{y_i-1})(1-\frac{L_{j-1}}{n_i}) 1-\frac{x^{**}_{y_i-1}}{n_i} \leq 1-\tilde{x}_{y_i-1}.$$
Combining these two relations we get as before that
$$R(1+2^{-(j+3)}) \frac{y^{**}_i-1}{n_i-x^{**}_{y_{i}-1}}\geq \frac{1-2^{-(j+3)}}{R}$$
for $L_0$ sufficiently large.

Thus we have constructed a corner to corner route in $A\otimes \mathcal{B}$. From the definitions it follows easily that for $P$ as above $V(P)\subseteq R(A)$ and hence proof of $(i)$ is complete.

Proof of $(ii)$ is similar. Say, for the side to corner admissible connection, for a given $b\in S_{in}$, in stead of starting with the line $y=(t'/t)x$, we start with the line passing through $(b_1/n_1,0)$ and $(t,t')$, and define $\tilde{x}_i$, $\tilde{y}_i$ to be the intersection of this line with the lines $y=i$ and $x=i$ respectively. Rest of the proof is almost identical, we use the fact $t,t'>5^{j+6}R$ to prove that the slope of this new line is still sufficiently close to $t'/t$.

For part $(iii)$, instead of a straight line we start with a number of piecewise linear functions which approximate $V(P)$. By taking a large number of such choices, it follows that for one of the cases $V(P)$ must be disjoint with the given set $A'$, we omit the details.
\end{proof}

Finally we show that if we try a large number of admissible assignments, at least one of them must obey the hypothesis in Lemma \ref{l:ccsuccessstar} and Lemma \ref{l:ccsssstar} regarding $A'$

\begin{lemma}
\label{l:admissiblemapexistence}
Assume the set-up in Proposition \ref{p:assign}. Let $\Upsilon_h, h\in [L_j^2]$ be the family of admissible assignments of $(I_1,I_2)$ w.r.t. $(B,B')$ described in Proposition \ref{p:assign}$(i)$. Fix any arbitray $\mathcal{T}\subset [L_j^2]$ with $|\mathcal{T}|=R^6k_0^510^{2j+20}$. Then for every $S\subset I_1\times I_2$ with $|S|=k_0$, there exist $h_0\in \mathcal{T}$ such that $$\min_{x\in B_X, y\in B_Y, s\in S}\{|(x,\tau_{h_0}(x))-s|, |(\tau_{h_0}^{-1}(y),y)|-s|\}\geq 2k_0R^310^{j+8}.$$
\end{lemma}

\begin{proof}
Call $(x,y)\in I_1\times I_2$ \emph{forbidden} if there exist $s\in S$ such that $|(x,y)-s| \leq 2k_0R^310^{j+8}$. For each $s\in S$, let $B_s\subset I_1\times I_2$ denote the set of vertices which are \emph{forbidden} because of $s$, i.e., $B_s=\{(x,y): |(x,y)-s|\leq 2k_0R^310^{j+8}\}$. Clearly $|B_s|\leq 10^{2j+18}k_0^2R^{6}$. So the total number of forbidden vertices is $\leq 10^{2j+18}k_0^3R^{6}$. Since $|B|, |B'|\leq k_0$, there exists $\mathcal{H}\subset \mathcal{T}$ with $|\mathcal{H}|=10^{2j+19}R^6k_0^4$ such that for all $x,x'\in B$, $x\neq x'$, $y,y'\in B'$, $y\neq y'$, $h_1,h_2\in \mathcal{H}$, we have $\tau_{h_1}(x)\neq \tau_{h_2}(x')$ and $\tau_{h_1}^{-1}(y)\neq \tau_{h_2}^{-1}(y')$.
Now for each $x\in B$ (resp. $y\in B'$), $(x,\tau_{h}(x))$ (resp. $(\tau_h^{-1}(y),y)$) can be \emph{forbidden} for at most $10^{2j+18}k_0^3R^{6}$ many different $h\in \mathcal{H}$. Hence,
$$\# \bigcup_{x\in B, y\in B'}\{h\in \mathcal{H}: (x,\tau_h(x)) ~\text{or}~(\tau_h^{-1}(y),y)~\text{is forbidden}\}\leq 2\times 10^{2j+18}R^6k_0^4< |\mathcal{H}|.$$
It follows that there exist $h_0\in \mathcal{H}$ which satisfies the condition in the statement of the lemma.
\end{proof}

\section{Length estimate}\label{s:length}

We shall qoute the following theorem directly from \cite{BS14}.
\begin{theorem}[Theorem 8.1, \cite{BS14}]
\label{lengthestimate}
Let $X$ be an $\mathbb{X}$ block at level $(j+1)$ we have that
\begin{equation}
\label{length2a}
\mathbb{E}[\exp (L_{j}^{-6}(|X|-(2-2^{-(j+1)})L_{j+1}))] \leq 1.
\end{equation}
and hence for $x \geq 0$,
\begin{equation}
\label{length2}
\mathbb{P}(|X|> ((2-2^{-(j+1)})L_{j+1}+xL_{j}^6))\leq e^{-x}.
\end{equation}
\end{theorem}

The proof is exactly the same as in \cite{BS14}.

\section{Corner to Corner estimate}\label{s:tailestimate}
In this section we prove the recursive tail estimate for the corner to corner connection probabilities.

\begin{theorem}\label{t:tail}
Assume that the inductive hypothesis holds up to level $j$. Let $X$ and $Y$ be random  $(j+1)$-level blocks according to $\mu^\mathbb{X}_{j+1}$ and $\mu^\mathbb{Y}_{j+1}$. Then
\begin{align*}
\mathbb{P}\left(\mathbb{P}(X \stackrel {c,c}{\longleftrightarrow}Y|X)\leq p\right)\leq p^{m_{j+1}} L_{j+1}^{-\beta},\quad
\mathbb{P}\left(\mathbb{P}(X \stackrel {c,c}{\longleftrightarrow}Y|Y)\leq p\right)\leq p^{m_{j+1}} L_{j+1}^{-\beta}
\end{align*}
for $p\leq \frac{3}{4}+2^{-(j+4)}$ and $m_{j+1}=m+2^{-(j+1)}$.
\end{theorem}


Due to the obvious symmetry between our $X$ and $Y$ bounds and for brevity all our bounds will be stated in terms of $X$ and $S^\mathbb{X}_{j+1}$ but will similarly hold for $Y$ and $S^\mathbb{Y}_{j+1}$. For the rest of this section we drop the superscript $\mathbb{X}$ and denote $S_{j+1}^{\mathbb{X}}$ (resp. $S_j^{\mathbb{X}}$) simply by $S_{j+1}$ (resp. $S_j$).

The block $X$ is constructed from an i.i.d. sequence of $j$-level blocks $X_1,X_2,\ldots$ conditioned on the event $X_i\in G_j^{\mathbb{X}}$ for $1\leq i \leq L_j^3$ as described in Section~\ref{s:prelim}.  The construction also involves a random variable $W_X\sim \mathrm{Geom}(L_j^{-4})$ and let $T_X$ denote the number of extra sub-blocks of $X$, that is the length of $X$ is $L_j^{\alpha-1}+2L_j^3+T_X$.  Let $K_X$ denote the number of bad sub-blocks of $X$. Let us also denote the position of bad subblock of $X$ and their neighbours by $\{\ell_{1}<\ell_{2}<\cdots < \ell_{K'_{X}}\}$, where $K'_{X}$ denotes the number of such blocks. Trivially, $K'_{X}\leq 3K_{X}$. We define $Y,\ldots,W_Y,T_Y$ and $K_Y$ similarly.The proof of Theorem~\ref{t:tail} is divided into 5 cases depending on the number of bad sub-blocks, the total number of sub-blocks of $X$ and how ``bad'' the sub-blocks are.

We note here that the proof of Theorem \ref{t:tail} follows along the same general line of argument as the proof of Theorem 7.1 in \cite{BS14}, with significant adaptations resulting from the specifics of the model and especially the difference in the definition of good blocks. As such this section is similar to Section 7 in \cite{BS14}.

We quote the following key lemma providing a bound for  the probability of blocks having large length, number of bad sub-blocks or small  $\prod_{i=1}^{K_X}S_j(X_{\ell_i})$ from \cite{BS14}

\begin{lemma}\label{l:totalSizeBound}[Lemma 7.3, \cite{BS14}]
For all $t',k',x\geq 0$ we have that
\[
\P\left[T_X \geq t', K_X \geq k', -\log \prod_{i=1}^{K_X}S_j(X_{\ell_i}) > x\right] \leq 2 L_j^{-\delta k' /4}\exp\left(-x m_{j+1} - \frac12 t' L_j^{-4} \right).
\]
\end{lemma}


We now proceed with the 5 cases we need to consider.

\subsection{Case 1}
The first case is the scenario where the blocks are of typical length, have few bad sub-blocks whose corner to corner corner to corner connection probabilities are not too small.  This case holds with high probability.

We define the event $\A{1}_{X,j+1}$ to be the set of $(j+1)$ level blocks such that
\[
\A{1}_{X,j+1} := \left\{X:T_X \leq \frac{R L^{\alpha-1}_j}{2}, K_X\leq k_0,  \prod_{i=1}^{K_X}S_j(X_{\ell_i}) > L_j^{-1/3} \right\}.
\]

The following Lemma is an easy corollary of Lemma \ref{l:totalSizeBound} and the choices of parameters, we omit the proof.

\begin{lemma}
\label{l:A1Size}
The probability that $X\in \A{1}_{X,j+1}$ is bounded below by
\[
\P[X\not\in \A{1}_{X,j+1}] \leq L_{j+1}^{-3\beta}.
\]
\end{lemma}

\begin{lemma}\label{l:A1Map}
We have that for all $X\in\A{1}_{X,j+1}$,
\begin{equation}\label{e:A1MapE1}
\P[X \cc Y\mid Y\in \A{1}_{Y,j+1}, X] \geq \frac34 + 2^{-(j+3)},
\end{equation}
\end{lemma}

\begin{proof}
Suppose that $X \in\A{1}_{X,j+1} $ with length  $L_j^{\alpha -1}+2L_j^3+T_X$. Let $B_X$ denote the location of bad subblocks of $X$. let $K'_{X}$ be the number of bad sub-blocks and their neighbours and let set of their locations be $B^*=\{\ell_1<\cdots <\ell_{K'_{X}}\}$. Notice that $K'_X\leq 3k_0$. We condition on $Y\in\A{1}_{Y,j+1}$ having no bad subblocks. Denote this conditioning by
\begin{align*}
\mathcal{F}=\{ Y\in \A{1}_{Y,j+1},T_Y,K_Y=0\}.
\end{align*}
Let $I_1=[L_j^{\alpha -1}+2L_j^3+T_X]$ and $I_2=[L_j^{\alpha -1}+2L_j^3+T_Y]$. By Proposition~\ref{p:assign}(i), we can find $L_j^2$ admissible assignments $\Upsilon_{h}$ at level $j$ w.r.t. $(B^*,\emptyset)$, with associated $\tau_h$ for $1\leq h \leq L_j^2$, such that $\tau_h(\ell_i)=\tau_1(\ell_i)+h-1$ and in particular each block $\ell_i$ is mapped to $L_j^2$ distinct sub-blocks.  Hence we get $\mathcal{H}\subset [L_j^2]$ of size $L_j<\lfloor L_j^2/ 9 k_0^2\rfloor$ so that for all $i_1\neq i_2$ and $h_1,h_2\in\mathcal{H}$ we have that $\tau_{h_1}(\ell_{i_1})\neq \tau_{h_2}(\ell_{i_2})$, that is that all the positions bad blocks and their neighbours are mapped to are distinct.

Our construction ensures that all $Y_{\tau_{h}(\ell_{i})}$ are uniformly chosen good $j$-blocks conditional on $\cf$ and since $S_j(X_{\ell_{i}})\geq L_j^{-1/3}$ we have that if $X_{\ell_i}\notin G_j^{\mathbb{X}}$,
\begin{equation}\label{e:A1MapA}
\P[X_{\ell_{i}} \cc Y_{\tau_{h}(\ell_{i})}\mid \cf] \geq S_j(X_{\ell_{i}}) - \P[Y_{\tau_{h}(\ell_{i})}\not\in G_j^{\mathbb{X}}] \geq \frac12 S_j(X_{\ell_{i}}).
\end{equation}
Also if $X_{\ell_i}\in G_j^{\mathbb{X}}$ then from the recursive estimates it follows that
$$\P[X_{\ell_{i}} \cc Y_{\tau_{h}(\ell_{i})}\mid \cf] \geq \frac34;$$
$$\P[X_{\ell_{i}} \cs Y_{\tau_{h}(\ell_{i})}\mid \cf] \geq \frac{9}{10};$$
$$\P[X_{\ell_{i}} \scc  Y_{\tau_{h}(\ell_{i})}\mid \cf] \geq \frac{9}{10}.$$
If $X_{\ell_i}\notin G_j^{\mathbb{X}}$, or, if neither $X_{\ell_i -1}$ nor $X_{\ell_i +1}$ is $\in G_j^{\mathbb{X}}$, let $\cd_{h,i}$ denote the event
$$\cd_{h,i}=\left\{X_{\ell_{i}} \cc Y_{\tau_{h}(\ell_{i})}\right\} .$$
If $X_{\ell_i},X_{\ell_i +1}\in G_j^{\mathbb{X}}$ then let
$\cd_{h,i}$ denote the event
$$\cd_{h,i}=\left\{X_{\ell_{i}} \cs Y_{\tau_{h}(\ell_{i})}\right\} .$$
If $X_{\ell_i},X_{\ell_i -1}\in G_j^{\mathbb{X}}$ then let
$\cd_{h,i}$ denote the event
$$\cd_{h,i}=\left\{X_{\ell_{i}} \scc Y_{\tau_{h}(\ell_{i})}\right\} .$$

Let $\cd_h$ denote the event
\[
\cd_h=\bigcap _{i=1}^{K'_X}\cd_{h,i}.
\]

Further, $\mathcal{S}$ denote the event
\[
\mathcal{S}=\left\{X_{k}\ssc Y_{k'} \forall k\in [L_j^{\alpha -1}+2L_j^3+T_X]\setminus \{\ell_1,\ldots ,\ell_{K_X}\}, \forall k' \in [L_j^{\alpha -1}+2L_j^3+T_Y]  \right\}.
\]

Also let
\[
\mathcal{C}_1=\left\{X_{1}\cs Y_{1}\right\}~\text{and}~\mathcal{C}_2=\left\{X_{L_j^{\alpha -1}+2L_j^3+T_X}\scc Y_{L_j^{\alpha -1}+2L_j^3+T_Y}\right\}.
\]

By Lemma \ref{l:ccconnect}, Lemma \ref{l:ccconnect2} and Lemma~\ref{l:ccsuccessstar} if  $\cup_{h\in \mathcal{H}}\cd_h, \mathcal{S}, \mathcal{C}_1,\mathcal{C}_2$ all hold then $X\cc Y$.  Conditional on $\cf$, for $h\in\ch$, the $\cd_h$, $\mathcal{C}_1$, $\mathcal{C}_2$ are independent and and by~\eqref{e:A1MapA} and the recursive estimates ,
\begin{equation}\label{e:A1MapB}
\P[\cd_h \mid \cf]  \geq 2^{-5k_0}3^{2k_0} L_j^{-1/3}.
\end{equation}
Hence
\begin{equation}\label{e:A1MapC}
\P[\cup_{h\in\ch}\cd_h \mid \cf]  \geq 1-\left(1-2^{-5k_0}3^{2k_0} L_j^{-1/3}\right)^{L_j}\geq  1-  L_{j+1}^{-3\beta}.
\end{equation}

It follows from the recursive estimates that
\begin{equation}\label{e:A1MapC1}
\P[\cup_{h\in\ch}\cd_h , \mathcal{C}_1, \mathcal{C}_2 \mid \cf]\geq \left(\frac{9}{10}\right)^2 \left(1-  L_{j+1}^{-3\beta} \right)
\end{equation}

Also a union bound using the recursive estimates at level $j$ gives
\begin{equation}\label{e:A1MapC2}
\P[\neg \mathcal{S} \mid \cf]\leq (1+\frac{R}{2})^2L_{j}^{2\alpha -2}L_j^{-2\beta}\leq L_j^{-\beta}.
\end{equation}

It follows that

\begin{equation}\label{e:A1MapC3}
\P[X\cc Y \mid \cf]\geq \P[\cup_{h\in\ch}\cd_h, \mathcal{C}_1, \mathcal{C}_2, \mathcal{S}]\geq  \left(\frac{9}{10}\right)^2 \left(1-  L_{j+1}^{-3\beta} \right)-L_j^{-\beta}.
\end{equation}
Hence
\begin{align*}
\P[X \cc Y\mid Y\in \A{1}_{Y,j+1}, X,T_Y] &\geq \P[X  \cc Y \mid \cf]\cdot \P[K_Y=0\mid Y\in \A{1}_{Y,j+1}, T_Y]\\
&\geq \left( \left(\frac{9}{10}\right)^2 \left(1-  L_{j+1}^{-3\beta} \right)-L_j^{-\beta}\right)\cdot \P[K_Y=0\mid Y\in \A{1}_{Y,j+1}, T_Y].
\end{align*}
Removing the conditioning on $T_Y$ we get

\begin{align*}
\P[X \cc Y\mid Y\in \A{1}_{Y,j+1}, X]
&\geq \left( \left(\frac{9}{10}\right)^2 \left(1-  L_{j+1}^{-3\beta} \right)-L_j^{-\beta}\right)\cdot \P[K_Y=0\mid Y\in \A{1}_{Y,j+1}]\\
&\geq \left( \left(\frac{9}{10}\right)^2 \left(1-  L_{j+1}^{-3\beta} \right)-L_j^{-\beta}\right)\cdot \left(1-L_{j+1}^{-3\beta} -2L_{j}^{-\delta /4}\right)\\
&\geq \frac{3}{4}+2^{-(j+1)}
\end{align*}
for large enough $L_0$, where the penultimate inequality follows from  Lemma \ref{l:totalSizeBound} and Lemma \ref{l:A1Size}. This completes the lemma.
\end{proof}

\begin{lemma}\label{l:A1final}
When $\frac12\leq p \leq \frac34 +2^{-(j+4)}$
\[
\P(S_{j+1}(X)\leq p)\leq p^{m_{j+1}} L_{j+1}^{-\beta}
\]
\end{lemma}

\begin{proof}
By Lemma~\ref{l:A1Size} and~\ref{l:A1Map} we have that
for all $X\in \A{1}_{X,j+1}$
\begin{equation}\label{eq:case1general}
\P[X\cc Y \mid X] \geq \P[Y\in \A{1}_{Y,j+1}]\P[X\cc Y \mid X, Y\in \A{1}_{Y,j+1}]\geq \frac{3}{4}+2^{-(j+4)}.
\end{equation}
Hence if $\frac12\leq p \leq \frac34 +2^{-(j+4)}$
\begin{align*}
\P(\P[X \cc Y\mid X]\leq p) &\leq \P[X \notin \A{1}_{X,j+1}]\\
&\leq L_{j+1}^{-3\beta} \leq 2^{-m_{j+1}}L_{j+1}^{-\beta}\leq p^{m_{j+1}}L_{j+1}^{-\beta}.
\end{align*}

\end{proof}

\subsection{Case 2}

The next case involves blocks which are not too long and do not contain too many bad sub-blocks but whose bad sub-blocks may be very bad in the since that corner to corner connection probabilities of those might be really small.
We define the class of blocks $\A{2}_{X,j+1}$  as
\[
\A{2}_{X,j+1} := \left\{X:T_X \leq \frac{R L^{\alpha-1}_j}{2}, K_X\leq k_0,  \prod_{i=1}^{K_X}S_j(X_{\ell_i}) \leq L_j^{-1/3} \right\}.
\]

\begin{lemma}\label{l:A2Map}
For $X\in \A{2}_{X,j+1}$,
\[
S_{j+1}(X) \geq \min\left\{\frac12, \frac{1}{10}\left(\frac{3}{4}\right)^{2k_0} L_j \prod_{i=1}^{K_X}S_j(X_{\ell_i}) \right\}
\]
\end{lemma}
\begin{proof}
Suppose that $X \in\A{2}_{X,j+1}$.  Let $\ce$ denote the event
\[
\ce=\{W_Y\leq  L^{\alpha-1}_j, T_Y=W_Y \}.
\]
Then by definition of $W_Y$, $\P[W_Y\leq L^{\alpha-1}_j] \geq 1 - (1-L_j^{-4})^{L_j^{\alpha-1}}\geq 9/10$ while by the definition of the block boundaries the event $T_Y=W_Y$ is equivalent to their being no bad sub-blocks amongst $Y_{L_j^3+L_j^{\alpha-1}+W_Y+1},\ldots,Y_{L_j^3+L_j^{\alpha-1}+W_Y+2L_j^3}$, that is that we don't need to extend the block because of bad sub-blocks. Hence $\P[T_Y=W_Y] \geq (1-L_j^{-\delta})^{2L_j^3} \geq 9/10$.  Combining these we have that
\begin{equation}\label{e:A2MapCEbound}
\P[\ce]\geq 8/10.
\end{equation}
By our block construction procedure, on the event $T_Y=W_Y$ we have that the blocks $Y_{L_j^3+1},\ldots,Y_{L_j^3+L_j^{\alpha-1}+T_Y}$ are uniform $j$-level blocks.

Define $I_1,I_2, B_X$ and $B^*$ as in the proof of Lemma~\ref{l:A1Map}. Also set $[L_{j}^{\alpha-1}+2L_{j}^3+T_X] \cap B_X=G_X$. Using Proposition~\ref{p:assign} again we can find $L_j^2$ level $j$ admissible assignments $\Upsilon_h$ of $(I_1,I_2)$ w.r.t. $(B^*,\emptyset)$  for $1\leq h \leq L_j^2$ with associated $\tau_h$. As in Lemma~\ref{l:A1Map} we can construct a subset $\mathcal{H}\subset [L_j^2]$ with $|\mathcal{H}|= L_j<\lfloor L_j^2/ 9k_0^2\rfloor$ so that for all $i_1\neq i_2$ and $h_1,h_2\in\mathcal{H}$ we have that $\tau_{h_1}(\ell_{i_1})\neq \tau_{h_2}(\ell_{i_2})$, that is that all the positions bad blocks are assigned to are distinct.  We will estimate the probability that one of these assignments work.

In trying out these $L_j$ different assignments there is a subtle conditioning issue since conditioned on an assignment not working (e.g., the event $X_{\ell_i}\cc Y_{\tau_{h}(\ell_i)}$ failing) the distribution of $Y_{\tau_{h}(l_i)}$ might change. As such we condition on an  event $\cd_h \cup \cg_h $ which holds with high probability.

If $X_{\ell_i}\notin G_j^{\mathbb{X}}$, or, if neither $X_{\ell_i -1}$ nor $X_{\ell_i +1}$ is $\in G_j^{\mathbb{X}}$, let $\cd_{h,i}$ denote the event
$$\cd_{h,i}=\left\{X_{\ell_{i}} \cc Y_{\tau_{h}(\ell_{i})}\right\} .$$
If $X_{\ell_i},X_{\ell_i +1}\in G_j^{\mathbb{X}}$ then let
$\cd_{h,i}$ denote the event
$$\cd_{h,i}=\left\{Y_{\tau_{h}(\ell_{i})}\in G_j^{\mathbb{Y}} X_{\ell_{i}} \cs Y_{\tau_{h}(\ell_{i})}~\text{and}~X_{k} \ssc Y_{\tau_{h}(\ell_i)} \forall k\in G_X \right\} .$$
If $X_{\ell_i},X_{\ell_i -1}\in G_j^{\mathbb{X}}$ then let
$\cd_{h,i}$ denote the event
$$\cd_{h,i}=\left\{Y_{\tau_{h}(\ell_{i})}\in G_j^{\mathbb{Y}}, X_{\ell_{i}} \scc Y_{\tau_{h}(\ell_{i})} ~\text{and}~X_{k} \ssc Y_{\tau_{h}(\ell_i)} \forall k\in G_X \right\} .$$

Let $\cd_h$ denote the event
\[
\cd_h=\bigcap _{i=1}^{K'_X}\cd_{h,i}.
\]

Further, let
\[
\cg_h=\left\{Y_{\tau_h(\ell_{i})}\in G_j^{\mathbb{Y}}~\text{and}~ Y_{\tau_h(\ell_{i})} \ssc X_k \hbox{ for } 1\leq i \leq K'_X, k\in G_X\right\}.
\]
Then it follows from the recursive estimates and since $\beta>\alpha+\delta+1$ that
\[
\P[\cd_h \cup \cg_h \mid X,\ce] \geq \P[ \cg_h \mid X,\ce] \geq 1-10k_0L_j^{-\delta}.
\]
and since they are conditionally independent given $X$ and $\ce$,
\begin{equation}\label{e:A2MapA}
\P[\cap_{h\in\ch}(\cd_h \cup \cg_h) \mid X,\ce] \geq  (1-10k_0L_{j}^{-\delta})^{L_j}\geq 9/10.
\end{equation}
Now
\[
\P[\cd_h\mid X,\ce,( \cd_h \cup \cg_h)] \geq \P[\cd_h\mid X,\ce] \geq  \left(\frac{3}{4}\right)^{2k_0}\prod_{i=1}^{K_X} S_j(X_{\ell_i})
\]
and hence
\begin{align}\label{e:A2MapB}
\P[\cup_{h\in\ch} \cd_h \mid X,\ce, \cap_{h\in\ch}(\cd_h \cup \cg_h)] &\geq 1-\left(1- \left(\frac{3}{4}\right)^{2k_0}\prod_{i=1}^{K_X} S_j(X_{\ell_i})\right)^{L_j}\nonumber\\
&\geq \frac{9}{10}  \wedge \frac14 \left(\frac{3}{4}\right)^{2k_0} L_j \prod_{i=1}^{K_X} S_j(X_{\ell_i})
\end{align}
since $1-e^{-x}\geq x/4\wedge 9/10$ for $x\geq 0$.
Furthermore, if
\[
\cm=\left\{\exists h_1\neq h_2 \in\ch: \cd_{h_1}\setminus \cg_{h_1}, \cd_{h_2}\setminus \cg_{h_2} \right\},
\]
then
\begin{align}\label{e:A2MapC}
\P[\cm \mid X,\ce, \cap_{h\in\ch}(\cd_h \cup \cg_h)]
&\leq {L_j \choose 2} \P[\cd_{h}\setminus \cg_{h}\mid X,\ce, \cap_{h\in\ch}(\cd_h \cup \cg_h)]^2\nonumber\\
&\leq {L_j \choose 2} \left(2\left(\frac{3}{4}\right)^{2k_0}\prod_{i=1}^{K_X} S_j(X_{\ell_i}) \wedge 2L_j^{-\delta}  \right)^2\nonumber\\
&\leq L_j^{-(\delta-2)}\left(\frac{3}{4}\right)^{2k_0} \prod_{i=1}^{K_X} S_j(X_{\ell_i}).
\end{align}
Let $\cj_I=\cj_{1}$ and $\cj_{F}=\cj_{L_j^{\alpha-1}+2L_j^3+T_Y}$ denote the events
\[
\cj_I=\left\{X_{1} \cs Y_{1} \hbox{ and } X_k\ssc Y_{1} \hbox{ for all } k\in G_X\right\};
\]
\[
\cj_{F}=\left\{X_{L_j^{\alpha-1}+2L_j^3+T_Y} \scc Y_{L_j^{\alpha-1}+2L_j^3+T_Y} \hbox{ and } X_k\ssc Y_{L_j^{\alpha-1}+2L_j^3+T_Y} \forall k\in G_X\right\}.
\]
For $k\in \{2,\ldots L_j^{\alpha -1}+2L_j^3+T_Y-1,\}\setminus \cup_{h\in\ch, 1\leq i \leq K'_X}\{\tau_h(\ell_i)\}$, let $\cj_k$ denote the event
\[
\cj_k=\left\{Y_{k}\in G_j^{\mathbb{Y}}, X_{k'}\ssc Y_{k} \hbox{ for all } k'\in G_X\right\}.
\]

Finally let
\[
\cj=\bigcap _{k\in [L_j^{\alpha -1}+2L_j^3+T_Y]\setminus \cup_{h\in\ch, 1\leq i \leq K'_X}\{\tau_h(\ell_i)\}}\cj_k.
\]
Then it follows from the recursive estimates and the fact that $\cj_k$ are conditionally independent that
\begin{equation}\label{e:A2MapD}
\P[\cj \mid X,\ce] \geq \left(\frac{9}{10}\right)^2\left(1- RL_j^{\alpha -1-\beta}\right)^{2L_j^{\alpha-1}} \geq 3/4.
\end{equation}

If $\cj,\cup_{h\in\ch} \cd_h$ and $\cap_{h\in\ch}(\cd_h \cup \cg_h)$ all hold and $\cm$ does not hold then we can find at least one $h\in\ch$ such that $\cd_h$ holds and $\cg_{h'}$ holds for all $h'\in\ch\setminus\{h\}$. Then by Lemma~\ref{l:ccsuccessstar} as before we have that $X\cc Y$. Hence by~\eqref{e:A2MapA}, \eqref{e:A2MapB}, \eqref{e:A2MapC}, and~\eqref{e:A2MapD} and the fact that $\cj$ is conditionally independent of the other events that
\begin{align*}
\P[X\cc Y\mid X,\ce]
&\geq \P[\cup_{h\in\ch} \cd_h, \cap_{h\in\ch}(\cd_h \cup \cg_h), \cj, \ \neg \cm \mid X,\ce] \nonumber\\
&= \P[\cj \mid X,\ce]\P[\cup_{h\in\ch} \cd_h, \ \neg \cm \mid X,\ce,\cap_{h\in\ch}(\cd_h \cup \cg_h)]
\P[\cap_{h\in\ch}(\cd_h \cup \cg_h)\mid X,\ce]\\
&\geq \frac{27}{40}\left[ \frac{9}{10}  \wedge \frac14 \left(\frac{3}{4}\right)^{2k_0} L_j \prod_{i=1}^{K_X} S_j(X_{\ell_i}) - L_j^{-(\delta-2)} \prod_{i=1}^{K_X} S_j(X_{\ell_i}) \right]\\
&\geq \frac{3}{5} \wedge \frac15 L_j \left(\frac{3}{4}\right)^{2k_0}\prod_{i=1}^{K_X} S_j(X_{\ell_i}).
\end{align*}
Combining with~\eqref{e:A2MapCEbound} we have that
\begin{align*}
\P[X\hookrightarrow Y\mid X]
&\geq \frac{1}{2} \wedge \frac1{10}\left(\frac{3}{4}\right)^{2k_0} L_j \prod_{i=1}^{K_X} S_j(X_{\ell_i}),
\end{align*}
which completes the proof.
\end{proof}

\begin{lemma}\label{l:A2Bound}
When $0<p< \frac12$,
\[
\mathbb{P}(X\in \A{2}_{X,j+1}, S_{j+1}(X)\leq p)\leq \frac15 p^{m_{j+1}} L_{j+1}^{-\beta}
\]
\end{lemma}

\begin{proof}
We have that
\begin{align}
\mathbb{P}(X\in \A{2}_{X,j+1}, S_{j+1}(X)\leq p) &\leq \P\left[\frac1{10} \left(\frac{3}{4}\right)^{2k_0}L_j \prod_{i=1}^{K_X}S_j(X_{\ell_i}) \leq p\right]\nonumber\\
&\leq  2\left(\frac{10 p}{L_j}\left(\frac{4}{3}\right)^{2k_0}\right)^{m_{j+1}} \leq \frac15 p^{m_{j+1}} L_{j+1}^{-\beta}
\end{align}
where the first inequality holds  by Lemma~\ref{l:A2Map}, the second by Lemma~\ref{l:totalSizeBound} and the third holds for large enough $L_0$  since $m_{j+1}>m>\alpha\beta$.
\end{proof}

\subsection{Case 3}

The third case allows for a greater number of bad sub-blocks.
The class of blocks $\A{3}_{X,j+1}$ is defined as
\[
\A{3}_{X,j+1} := \left\{X:T_X \leq \frac{R L^{\alpha-1}_j}{2}, k_0\leq K_X\leq \frac{L_j^{\alpha-1}+T_X}{10 R_j^+} \right\}.
\]

\begin{lemma}\label{l:A3Map}
For $X\in \A{3}_{X,j+1}$,
\[
S_{j+1}(X) \geq \frac12 \left(\frac{3}{4}\right)^{2K_X} \prod_{i=1}^{K_X}S_j(X_{\ell_i})
\]
\end{lemma}

\begin{proof}
For this proof we only need to consider a single admissible assignment $\Upsilon$. Suppose that $X \in\A{3}_{X,j+1}$.  Again let $\ce$ denote the event
\[
\ce=\{W_Y\leq  L^{\alpha-1}_j, T_Y=W_Y \}.
\]
Similarly to~\eqref{e:A2MapCEbound} we have that,
\begin{equation}\label{e:A3MapCEbound}
\P[\ce]\geq 8/10.
\end{equation}
As before we have, on the event $T_Y=W_Y$, the blocks $Y_{L_j^3+1},\ldots,Y_{L_j^3+L_j^{\alpha-1}+T_Y}$ are uniform $j$-blocks since the block division did not evaluate whether they are good or bad.

Set $I_1,I_2,B_X,G_X$ and $B^{*}$ as in the proof of Lemma \ref{l:A2Map}.
By Proposition~\ref{p:assign} we can find a level $j$ admissible assignment $\Upsilon$ of $(I_1,I_2)$ w.r.t. $(B^*,\phi)$ with associated $\tau$ so that for all $i$, $L_j^3+1\leq \tau_h(\ell_i) \leq L_j^3+L_j^{\alpha-1}+T_Y$.  We estimate the probability that this assignment works.

If $X_{\ell_i}\notin G_j^{\mathbb{X}}$, or, if neither $X_{\ell_i -1}$ nor $X_{\ell_i +1}$ is $\in G_j^{\mathbb{X}}$, let $\cd_{i}$ denote the event
$$\cd_{i}=\left\{X_{\ell_{i}} \cc Y_{\tau(\ell_{i})}\right\} .$$
If $X_{\ell_i},X_{\ell_i +1}\in G_j^{\mathbb{X}}$ then let
$\cd_{i}$ denote the event
$$\cd_{i}=\left\{Y_{\tau(\ell_{i})}\in G_j^{\mathbb{Y}}, X_{\ell_{i}} \cs Y_{\tau(\ell_{i})}~\text{and}~X_{k}\ssc Y_{\tau(\ell_{i})} \forall k\in G_X\right\}.$$
If $X_{\ell_i},X_{\ell_i -1}\in G_j^{\mathbb{X}}$ then let
$\cd_{i}$ denote the event
$$\cd_{i}=\left\{Y_{\tau(\ell_{i})}\in G_j^{\mathbb{Y}}, X_{\ell_{i}} \scc Y_{\tau(\ell_{i})}~\text{and}~X_{k}\ssc \} Y_{\tau(\ell_{i})} \forall k\in G_X\right\} .$$

Let $\cd$ denote the event
\[
\cd=\bigcap _{i=1}^{K'_X}\cd_{i}.
\]

By definition and the recursive estimates,
\begin{align}\label{e:A3MapA}
\P[\cd\mid X,\ce]  \geq \left(\frac{3}{4}\right)^{2K_X}\prod_{i=1}^{K_X} S_j(X_{\ell_i})
\end{align}

Let $\cj_I=\cj_{1}$ and $\cj_{F}=\cj_{L_j^{\alpha-1}+2L_j^3+T_Y}$ denote the events
\[
\cj_I=\left\{X_{1} \cs Y_{1} \hbox{ and } X_k\ssc Y_{1} \hbox{ for all } k\in G_X\right\};
\]
\[
\cj_{F}=\left\{X_{L_j^{\alpha-1}+2L_j^3+T_Y} \scc Y_{L_j^{\alpha-1}+2L_j^3+T_Y} \hbox{ and } X_k\ssc Y_{L_j^{\alpha-1}+2L_j^3+T_Y} \forall k\in G_X\right\}.
\]
For $k\in \{2,\ldots L_j^{\alpha -1}+2L_j^3+T_Y-1,\}\setminus \cup_{1\leq i \leq K'_X}\{\tau(\ell_i)\}$, let $\cj_k$ denote the event
\[
\cj_k=\left\{Y_k \in G_j^{\mathbb{Y}}, X_{k'}\ssc Y_{k} \hbox{ for all } k'\in G_X\right\}.
\]

Finally let
\[
\cj=\bigcap _{k\in [L_j^{\alpha -1}+2L_j^3+T_Y]\setminus \cup_{1\leq i \leq K'_X}\{\tau(\ell_i)\}}\cj_k.
\]

From the recursive estimates
\begin{equation}\label{e:A3MapB}
\P[\cj \mid X,\ce] \geq \frac{3}{4}.
\end{equation}

If $\cd$ and $\cj$ hold then by Lemma~\ref{l:ccsuccessstar} we have that $X\cc Y$. Hence by~\eqref{e:A3MapA} and \eqref{e:A3MapB} and the fact that $\cd$ and $\cj$ are conditionally independent we have that,
\begin{align*}
\P[X\cc Y\mid X,\ce]
&\geq \P[\cd, \cj \mid X,\ce] \nonumber\\
&= \P[\cd \mid X,\ce] \P[\cj \mid X,\ce] \nonumber\\
&\geq \frac{3}{4} \left(\frac{3}{4}\right)^{2K_X}\prod_{i=1}^{K_X} S_j(X_{\ell_i}).
\end{align*}
Combining with~\eqref{e:A3MapCEbound} we have that
\begin{align*}
\P[X\cc Y\mid X]
&\geq \frac{1}{2} \left(\frac{3}{4}\right)^{2K_X}\prod_{i=1}^{K_X} S_j(X_{\ell_i}),
\end{align*}
which completes the proof.
\end{proof}

\begin{lemma}\label{l:A3Bound}
When $0<p\leq \frac12$,
\[
\mathbb{P}(X\in \A{3}_{X,j+1}, S_{j+1}(X)\leq p)\leq \frac15 p^{m_{j+1}} L_{j+1}^{-\beta}
\]
\end{lemma}

\begin{proof}
We have that
\begin{align}
\mathbb{P}(X\in \A{3}_{X,j+1}, S_{j+1}(X)\leq p) &\leq \P\left[K_X>k_0, \frac12 \left(\frac{3}{4}\right)^{2K_X}\prod_{i=1}^{K_X}S_j(X_{\ell_i}) \leq p\right]\nonumber\\
&\leq \sum_{k=k_0}^{\infty}\P\left[K_X=k, \prod_{i=1}^{K_X}S_j(X_{\ell_i}) \leq 2p\left(\frac{4}{3}\right)^{2k}\right]\nonumber\\
&\leq  2\sum_{k=k_0}^{\infty}\left(2p\left(\frac{4}{3}\right)^{2k}\right)^{m_{j+1}} L_j^{-\delta k/4}\leq \frac15 p^{m_{j+1}} L_{j+1}^{-\beta}
\end{align}
where the first inequality holds  by Lemma~\ref{l:A3Map}, the third follows from Lemma \ref{l:totalSizeBound} and the last one holds for large enough $L_0$  since $\delta k_0 >4\alpha\beta$.
\end{proof}

\subsection{Case 4}

In Case 4 we allow blocks of long length but not too many bad sub-blocks.
The class of blocks $\A{4}_{X,j+1}$ is defined as
\[
\A{4}_{X,j+1} := \left\{X:T_X > \frac{R L^{\alpha-1}_j}{2}, K_X\leq \frac{L_j^{\alpha-1}+T_X}{10 R_j^+} \right\}.
\]

\begin{lemma}\label{l:A4Map}
For $X\in \A{4}_{X,j+1}$,
\[
S_{j+1}(X) \geq  \left(\frac{3}{4}\right)^{2K_X}\prod_{i=1}^{K_X}S_j(X_{\ell_i}) \exp(-3T_X L_j^{-4} /R )
\]
\end{lemma}

\begin{proof}
In this proof we  allow the length of $Y$ to grow at a slower rate than that of $X$.  Suppose that $X \in\A{4}_{X,j+1}$ and let $\ce(X)$ denote the event
\[
\ce(X)=\{W_Y = \lfloor 2 T_X / R\rfloor, T_Y=W_Y \}.
\]
Then by definition $\P[W_Y = \lfloor 2 T_X / R\rfloor] =L_j^{-4} (1-L_j^{-4})^{\lfloor 2 T_X / R\rfloor}$.  Similarly to Lemma~\ref{l:A2Map}, $\P[T_Y=W_Y\mid W_Y] \geq (1-L_j^{-\delta})^{2L_j^3} \geq 9/10$.  Combining these we have that
\begin{equation}\label{e:A4MapCEbound}
\P[\ce(X)]\geq \frac9{10} L_j^{-4} (1-L_j^{-4})^{\lfloor 2 T_X / R\rfloor}.
\end{equation}

Set $I_1,I_2, B_X,B^*$ as before. By Proposition~\ref{p:assign} we can find an admissible assignment at level $j$, $\Upsilon$ of $(I_1,I_2)$ w.r.t. $(B^*,\emptyset)$  with associated $\tau$ so that for all $i$, $L_j^3+1\leq \tau(\ell_i) \leq L_j^3+L_j^{\alpha-1}+T_Y$.  We again estimate the probability that this assignment works.

We need to modify the definition of $\cd$ and $\cj$ in this case since the length of $X$ could be arbitrarily large. For $k\in [L_j^{\alpha-1}+2L_{j}^3+T_Y]\setminus \tau(B_X)$, let $H^{\tau}_k\subseteq [L_j^{\alpha-1}+2L_{j}^3+T_Y]\setminus B_X$ be the sets given by Lemma \ref{l:ccsuccessstar} such that $|H^{\tau}_k|\leq L_j$ and there exists a $\tau$-compatible admissible route with $k$-sections contained in $H^{\tau}_{k}$ for all $k$. We define $\cd$ and $\cj$ in this case as follows.

If $X_{\ell_i}\notin G_j^{\mathbb{X}}$, or, if neither $X_{\ell_i -1}$ nor $X_{\ell_i +1}$ is $\in G_j^{\mathbb{X}}$, let $\cd_{i}$ denote the event
$$\cd_{i}=\left\{X_{\ell_{i}} \cc Y_{\tau(\ell_{i})}\right\} .$$
If $X_{\ell_i},X_{\ell_i +1}\in G_j^{\mathbb{X}}$ then let
$\cd_{i}$ denote the event
$$\cd_{i}=\left\{Y_{\tau(\ell_{i})}\in G_j^{\mathbb{Y}}, X_{\ell_{i}} \cs Y_{\tau(\ell_{i})}~\text{and}~X_{k}\ssc Y_{\tau(\ell_{i})} \forall k\in H_{\tau(l_i)}^{\tau}\right\}.$$
If $X_{\ell_i},X_{\ell_i -1}\in G_j^{\mathbb{X}}$ then let
$\cd_{i}$ denote the event
$$\cd_{i}=\left\{Y_{\tau(\ell_{i})}\in G_j^{\mathbb{Y}}, X_{\ell_{i}} \scc Y_{\tau(\ell_{i})}~\text{and}~X_{k}\ssc \} Y_{\tau(\ell_{i})} \forall k\in H_{\tau(l_i)}^{\tau}\right\} .$$

Let $\cd$ denote the event
\[
\cd=\bigcap _{i=1}^{K'_X}\cd_{i}.
\]

Let $\cj_I=\cj_{1}$ and $\cj_{F}=\cj_{L_j^{\alpha-1}+2L_j^3+T_Y}$ denote the events
\[
\cj_I=\left\{X_{1} \cs Y_{1} \hbox{ and } X_k\ssc Y_{1} \hbox{ for all } k\in H_{1}^{\tau}\right\};
\]
\[
\cj_{F}=\left\{X_{L_j^{\alpha-1}+2L_j^3+T_Y} \scc Y_{L_j^{\alpha-1}+2L_j^3+T_Y} \hbox{ and } X_k\ssc Y_{L_j^{\alpha-1}+2L_j^3+T_Y} \forall k\in H_{L_j^{\alpha-1}+2L_j^3+T_Y}^{\tau}\right\}.
\]
For $k\in \{2,\ldots L_j^{\alpha -1}+2L_j^3+T_Y-1,\}\setminus \cup_{1\leq i \leq K'_X}\{\tau(\ell_i)\}$, let $\cj_k$ denote the event
\[
\cj_k=\left\{Y_k \in G_j^{\mathbb{Y}}, X_{k'}\ssc Y_{k} \hbox{ for all } k'\in H_k^{\tau}\right\}.
\]

Finally let
\[
\cj=\bigcap _{k\in [L_j^{\alpha -1}+2L_j^3+T_Y]\setminus \cup_{1\leq i \leq K'_X}\{\tau(\ell_i)\}}\cj_k.
\]

If $\cd$ and $\cj$ hold then by Lemma~\ref{l:ccsuccessstar} we have that $X\cc Y$. It is easy to see that, in this case~\eqref{e:A3MapA} holds. Also we have for large enough $L_0$,
\begin{equation}\label{e:A4MapB}
\P[\cj \mid X,\ce(X)] \geq \frac{3}{4}\left(1- 2L_j^{-\delta}\right)^{L_j^{\alpha-1} + \lfloor 2 T_X / R\rfloor + 2L_j^3} \geq \frac14\exp\left(- 2L_j^{-\delta}(L_j^{\alpha-1} + \lfloor 2 T_X / R\rfloor + 2L_j^3)\right).
\end{equation}

Hence by~\eqref{e:A3MapA} and \eqref{e:A4MapB} and the fact that $\cd$ and $\cj$ are conditionally independent we have that,
\begin{align*}
\P[X\cc Y\mid X,\ce]
&\geq \P[\cd \mid X,\ce] \P[\cj \mid X,\ce] \nonumber\\
&\geq \frac14\exp\left(- L_j^{-\delta}(L_j^{\alpha-1} + \lfloor 2 T_X / R\rfloor + 2L_j^3)\right) \left(\frac{3}{4}\right)^{2K_X}\prod_{i=1}^{K_X} S(X_{\ell_i}).
\end{align*}
Combining with~\eqref{e:A4MapCEbound} we have that
\begin{align*}
\P[X\cc Y\mid X]
&\geq  \exp(-3T_X L_j^{-4} /R ) \left(\frac{3}{4}\right)^{2K_X} \prod_{i=1}^{K_X} S(X_{\ell_i}),
\end{align*}
since $T_X L_j^{-4} = \Omega(L_j^{\alpha-6})$ and $\delta >5$ which completes the proof.
\end{proof}

\begin{lemma}\label{l:A4Bound}
When $0<p\leq \frac12$,
\[
\mathbb{P}(X\in \A{4}_{X,j+1}, S_{j+1}(X)\leq p)\leq \frac15 p^{m_{j+1}} L_{j+1}^{-\beta}
\]
\end{lemma}

\begin{proof}
We have that
\begin{align}
\mathbb{P}(X\in \A{4}_{X,j+1}, S_{j+1}(X)\leq p) &\leq \sum_{t=\frac{R L^{\alpha-1}_j}{2}+1}^\infty \sum_{k=k_0}^{\infty} \P\left[T_X=t, K_X=k,   \left(\frac{3}{4}\right)^{2k}\prod_{i=1}^{K_X}S_j(X_{\ell_i}) \exp(-3 t L_j^{-4} /R ) \leq p\right]\nonumber\\
&\leq \sum_{t=\frac{R L^{\alpha-1}_j}{2}+1}^\infty \sum_{k=k_0}^{\infty} 2\left(p  \left(\frac{4}{3}\right)^{2k}\exp(3 t L_j^{-4} /R )\right)^{m_{j+1}}L_j^{-\delta k/4} \exp\left(- \frac12 t L_j^{-4}\right)\nonumber\\
&\leq \frac15 p^{m_{j+1}} L_{j+1}^{-\beta}
\end{align}
where the first inequality holds  by Lemma~\ref{l:A4Map}, the second by Lemma~\ref{l:totalSizeBound} and the third holds for large enough $L_0$  since $3m_{j+1}/R<\frac12$ and so for large enough $L_0$, $(4/3)^{2(m+1)}L_j^{-\delta/4} \leq 1/2$ and
\[
\sum_{t=R L^{\alpha-1}_j/2+1}^\infty \exp\left(- t L_j^{-4} \left(\frac12-\frac{3m_{j+1}}{R} \right)\right) < \frac1{10} L_{j+1}^{-\beta}.
\]
\end{proof}

\subsection{Case 5}
It remains to deal with the case involving blocks with a large density of bad sub-blocks. Define the class of blocks $\A{5}_{X,j+1}$ is as
\[
\A{5}_{X,j+1} := \left\{X: K_X > \frac{L_j^{\alpha-1}+T_X}{10 R_j^+} \right\}.
\]

\begin{lemma}\label{l:A5Map}
For $X\in \A{5}_{X,j+1}$,
\[
S_{j+1}(X) \geq \exp(-2 T_X L_j^{-4})\left(\frac{3}{4}\right)^{2K_X} \prod_{i=1}^{K_X} S_j(X_{\ell_i})
\]
\end{lemma}
\begin{proof}
The proof is a minor modification of the proof of Lemma~\ref{l:A4Map}.  We take $\ce(X)$ to denote the event
\[
\ce(X)=\{W_Y = T_X, T_Y=W_Y \}.
\]
and get a bound of
\begin{equation}\label{e:A5MapCEbound}
\P[\ce(X)]\geq \frac{9}{10} L_j^{-4} (1-L_j^{-4})^{T_X}.
\end{equation}
We consider the admissible assignment $\Upsilon$ given by $\tau(i)=i$ for $i\in B^*$. It follows from Lemma \ref{l:ccsuccessstar} that in this case we can define $H_k^{\tau}={k-1,k,k+1}$. We define $\cd$ and $\cj$ as before. The new bound for $\cj$ becomes
\begin{equation}\label{e:A5MapB}
\P[\cj \mid X,\ce(X)] \geq \frac34 \left(1- 2L_j^{-\delta}\right)^{L_j^{\alpha-1} + T_X + 2L_j^3} \geq \frac14\exp\left(-2L_j^{-\delta}(L_j^{\alpha-1} + T_X + 2L_j^3)\right).
\end{equation}
We get the result proceeding as in the proof of Lemma~\ref{l:A4Map}.
\end{proof}

\begin{lemma}\label{l:A5Bound}
When $0<p\leq \frac12$,
\[
\mathbb{P}(X\in \A{5}_{X,j+1}, S_{j+1}(X)\leq p)\leq \frac15 p^{m_{j+1}} L_{j+1}^{-\beta}
\]
\end{lemma}

\begin{proof}
First note that since $\alpha>4$, $$L_j^{-\frac{\delta}{50 R_j^+}}=L_0^{-\frac{\delta \alpha^j}{50 R_j^+}}\to 0$$ as $j\to\infty$.  Hence for large enough $L_0$,
\begin{equation}\label{e:A5BoundA}
\sum_{t=0}^\infty\left(\exp(2 m_{j+1} L_j^{-4}) L_j^{-\frac{\delta}{50 R_j^+}} \right)^t < 2.
\end{equation}
We have that
\begin{eqnarray}
\mathbb{P}(X\in \A{5}_{X,j+1}, S_{j+1}(X)\leq p)
&\leq & \sum_{t=0}^\infty\sum_{k=\frac{L_j^{\alpha-1}+t}{10 R_j^+}}^{\infty} \P\left[T_X=t, K_X=k, \left(\frac{3}{4}\right)^{2k}\prod_{i=1}^{K_X}S_j(X_{\ell_i}) \exp(-2 t L_j^{-4}) \leq p\right]\nonumber\\
&\leq & p^{m_{j+1}} \sum_{t=0}^\infty \sum_{k=\frac{L_j^{\alpha-1}+t}{10 R_j^+}}^{\infty} 2\left( \exp(2 m_{j+1} t L_j^{-4})\right)\left( \left(\frac{16}{9}\right)^{m_j+1}L_j^{-\frac{\delta}{4}}\right)^{k}\nonumber\\
&\leq & p^{m_{j+1}} \sum_{t=0}^\infty 4 \left( \exp(2 m_{j+1} t L_j^{-4})\right)L_j^{-\frac{L_j^{\alpha-1}+t}{50 R_j^+}}\nonumber\\
&\leq & \frac15 p^{m_{j+1}} L_{j+1}^{-\beta}
\end{eqnarray}
where the first inequality holds be by Lemma~\ref{l:A5Map}, the second by Lemma~\ref{l:totalSizeBound} and the third follows since $L_0$ is sufficiently large and the last one by~\eqref{e:A5BoundA} and the fact that
\[
L_j^{-\frac{\delta L_j^{\alpha-1}}{50 R_j^+}} \leq \frac1{40} L_{j+1}^{-\beta},
\]
for large enough $L_0$.
\end{proof}

\subsection{Proof of Theorem~\ref{t:tail}}
Putting together all the five cases we now prove Theorem \ref{t:tail}.

\begin{proof}[Proof of Theorem~\ref{t:tail}]
The case of $\frac12\leq p \leq 1-L_{j+1}^{-1}$ is established in Lemma~\ref{l:A1final}.  By Lemma~\ref{l:A1Map} we have that $S_{j+1}(X) \geq \frac12$ for all $X\in \A{1}_{X,j+1}$.  Hence we need only consider $0<p<\frac12$ and cases 2 to 5.  By Lemmas~\ref{l:A2Bound}, \ref{l:A3Bound}, \ref{l:A4Bound} and~\ref{l:A5Bound} then
\begin{align*}
\mathbb{P}(S_{j+1}(X)\leq p) \leq \sum_{l=2}^5 \mathbb{P}(X\in \A{l}_{X,j+1}, S_{j+1}(X)\leq p)\leq  p^{m_{j+1}} L_{j+1}^{-\beta}.
\end{align*}
The bound for $S_{j+1}^{\mathbb{Y}}$ follows similarly.
\end{proof}

\section{Side to Corner and Corner to Side estimates}\label{s:scestimate}
The aim of this section is to show that for a large class of $\mathbb{X}$- blocks (resp. $\mathbb{Y}$-blocks), $\P(X\cs Y\mid X)$ and $\P(X\scc Y\mid X)$ (resp. $\P(X\cs Y\mid Y)$ and $\P(X\scc Y\mid Y)$) is large. We shall state and prove the result only for $\mathbb{X}$-blocks.

Here we need to consider a different class of blocks where the blocks have few bad sub-blocks whose corner to corner connection probabilities are not too small, where the excess number of subblocks is of smaller order than the typical length and none of the subblocks, and their chunks contain too many level $0$ blocks.  This case holds with high probability. Let $X$ be a level $(j+1)$ $X$-block constructed out of the independent sequence of $j$ level blocks $X_1,X_2,\ldots $ where the first $L_j^3$ ones are conditioned to be good.

For $i=1,2,\ldots , L_j^{\alpha-1}+2L_j^{3}+T_X$, let $\mathcal{G}_i$ denote the event that all level $j-1$ subblocks contained in $X_i$ contains at most $3L_{j-1}$ level $0$ blocks, and $X_i$ contains at most $3L_j$ level $0$ blocks. Let $\mathcal{G}_{X}$ denote the event that for all good blocks $X_i$ contained in $X$, $\mathcal{G}_{i}$ holds. We define $\A{*}_{X,j+1}$ to be the set of $(j+1)$ level blocks such that
\[
\A{*}_{X,j+1} := \left\{X:T_X \leq L_j^{5}-2L_j^3, K_X\leq k_0,  \prod_{i=1}^{K_X}S_j(X_{\ell_i}) > L_j^{-1/3}, \mathcal{G}_{X}\right\}.
\]

It follows from Theorem \ref{lengthestimate} that $\P[\mathcal{G}_{X}^{c}]$ is exponentially small in $L_{j-1}$ and hence we shall be able to safely ignore this conditioning while calculating probability estimates since $L_0$ is sufficiently large.
%
%
%
%
%

%
%
%

Similarly to Lemma \ref{l:A1Size} it can be proved that
\begin{equation}
\label{e:ASSize}
\P[X \in \A{*}_{X,j+1}]\geq 1-L_{j+1}^{-3\beta}.
\end{equation}

We have the following proposition.

\begin{proposition}\label{l:SCMap}
We have that for all $X\in\A{*}_{X,j+1}$,
\begin{equation}\label{e:SCMapE1}
\P[X \cs Y\mid Y\in \A{*}_{Y,j+1}, X] \geq \frac{9}{10} + 2^{-(j+15/4)},~\P[X \scc Y\mid Y\in \A{1}_{Y,j+1}, X] \geq \frac{9}{10} + 2^{-(j+15/4)}.
\end{equation}
\end{proposition}

We shall only prove the corner to side estimate, the other one follows by symmetry. Suppose that $X \in\A{*}_{X,j+1} $ with length  $L_j^{\alpha -1}+2L_j^3+T_X$, define $B_X$, $B^{*}$,$K'_{X}$, $T_Y$ and $K_Y$ as in the proof of Lemma \ref{l:A1Map}.
We condition on $Y\in\A{*}_{Y,j+1}$ having no bad subblocks. Denote this conditioning by
\begin{align*}
\mathcal{F}=\{ Y\in \A{*}_{Y,j+1},T_Y,K_Y=0\}.
\end{align*}

Let $n_X$ and $n_{Y}$ denote the number of chunks in $X$ and $Y$ respectively. We first prove the following lemma.
\begin{lemma}
\label{l:scmap1}
Consider an exit chunk $(k,n_{Y})$ (resp. $(n_X,k)$) in $\mathcal{E}_{out}(X,Y)$. Fix $t\in [L_{j}^{\alpha-1}+2L_{j}^3+T_X]$ contained in $C_k^{X}$ such that $[t,t-L_{j}^3]\cap B_X=\emptyset$ (resp. fix $t'\in [L_{j}^{\alpha-1}+2L_{j}^3+T_Y]$ contained in $C_k^{Y}$). Consider $\tilde{X}=(X_1,\ldots, X_t)$ (or $\tilde{Y}=(Y_1,\ldots, Y_{t'})$). Then there exists an event $\mathcal{S}_t$ with $\P[S_t\mid \cf]\geq 1-L_{j}^{-\alpha}$  and on $S_t$, $\cf$ and $\{X_1\cs Y_1\}$ we have
$\tilde{X}\cso Y$ (resp. $\mathcal{S}_{t'}$ with $\P[S_{t'}\mid \cf]\geq 1-L_{j}^{-\alpha}$  and on $S_{t'}$, $\cf$ and $\{X_1\cs Y_1\}$ we have $X\cso \tilde{Y}$).
\end{lemma}

%
\begin{proof}
We shall only prove the first case, the other case follows by symmetry.
Set $I_1=[t]$, $I_2=[L_j^{\alpha-1}+2L_j^{3}+T_Y]$. Also define $B_{\tilde{X}}$ and $B^*$ as in the proof of Lemma \ref{l:A1Map}. The slope condition in the definition of $\mathcal{E}_{out}(X,Y)$, and the fact that $B_X$ is disjoint with $[t-L_j^3,t]$ implies that by Proposition~\ref{p:assign} we can find $L_j^2$ admissible generalized mappings $\Upsilon_h$ of $(I_1,I_2)$ with respect to $(B^*,\emptyset)$  with associated $\tau_h$ for $1\leq h \leq L_j^2$ as in the proof of Lemma \ref{l:A1Map}.
As in there, we construct a subset $\mathcal{H}\subset [L_j^2]$ with $|\mathcal{H}|= L_j<\lfloor L_j^2/ 3 k_0\rfloor$ so that for all $i_1\neq i_2$ and $h_1,h_2\in\mathcal{H}$ we have that $\tau_{h_1}(\ell_{i_1})\neq \tau_{h_2}(\ell_{i_2})$.


For $h\in \mathcal{H},i\in B^*$, define the events $\cd_{h,i}$ similarly as in the proof of Lemma \ref{l:A1Map}. Set
%

$$\cd_h=\bigcap _{i=1}^{K'_X}\cd^k_{h,i} ~\text{and}~
\cd=\bigcup _{h\in \ch}\cd^k_{h}.$$

Further, $\mathcal{S}$ denote the event
\[
\mathcal{S}=\left\{X_{k}\ssc Y_{k'} \forall k\in [t]\setminus \{\ell_1,\ldots ,\ell_{K'_X}\}, \forall k' \in [L_j^{\alpha -1}+2L_j^3+T_Y]  \right\}.
\]


Same arguments as in the proof of yields
\begin{equation}\label{e:ASMapC}
\P[\cd \mid \cf] \geq  1-  L_{j+1}^{-3\beta}
\end{equation}
and
\begin{equation}\label{e:ASMapC2}
\P[\neg \mathcal{S} \mid \cf]\leq 4L_{j}^{2\alpha -2}L_j^{-2\beta}\leq L_j^{-\beta}.
\end{equation}

Now it follows from Lemma \ref{l:csscconnect} and Lemma \ref{l:ccsssstar}, that on $\{X_1\cs Y_1\}$, $\mathcal{S}, \cd$ and $\cf$, we have $\tilde{X}\cso Y$. The proof of the Lemma is completed by setting $\mathcal{S}_t=\mathcal{S}\cap \cd$.
\end{proof}
%

Now we are ready to prove Proposition \ref{l:SCMap}.

\begin{proof}[Proof of Proposition \ref{l:SCMap}]
Fix an exit chunk $(k,n_Y)$ or $(n_X,k')$ in $\mathcal{E}_{out}(X,Y)$. In the former case set $T_k$ to be the set of all blocks $X_t$ contained in $C_k^{X}$  such that $[t,t-L_j^3]\cap B_X=\emptyset$, in the later case set $T'_{k'}$ to be the set of all blocks $Y_{t'}$ conttained in $C_{k'}^{Y}$. Notice that the number of blocks contained in $T_k$ is at least $(1-2k_0L_j^{-1})$ fraction of the total number of blocks contained in $C_k^{X}$. For $t\in T_k$ (resp. $t'\in T'_{k'}$), let $S_t$ (resp. $S_{t'}$) be the event given by Lemma \ref{l:scmap1}
Hence it follows from Lemma \ref{l:csssconnect}(i), that on
$\{X_1 \cs Y_1\}\bigcap  \cap_{k,T_k} \mathcal{S}_{t} \bigcap \cap_{k', T_{k'}} \mathcal{S}_{t'}$, we have $X\cs Y$. Taking a union bound and using Lemma \ref{l:scmap1} and also using the recursive lower bound on $\P[X_1\cs Y_1]$ yields,

$$\P[X\cs Y\mid \cf, X]\geq \frac{9}{10}+ 2^{-(j+31/8)}.$$
The proof can now be completed by removing the conditioning on $T_Y$ and proceeding as in Lemma \ref{l:A1Map}.
\end{proof}

\section{Side to Side Estimate}\label{s:ssestimate}
In this section we estimate the probability of having a side to side path in $X\times Y$. We work in the set up of previous section. We have the following theorem.

\begin{proposition}\label{l:SSMap}
We have that
\begin{equation}\label{e:SSMapE2}
\P[X \ssc Y\mid X \in\A{*}_{X,j+1} , Y\in \A{*}_{Y,j+1}] \geq 1-  L_{j+1}^{-3\beta}.
\end{equation}
\end{proposition}

Suppose that $X \in\A{*}_{X,j+1} , Y\in \A{*}_{Y,j+1}$. Let $T_X$, $T_{Y}$, $B_X$, $B_Y$, $G_X$, $G_Y$ as before. Let $B_1^*=\{\ell_1<\cdots< \ell_{K'_X}\}$ and $B_2^*=\{\ell_1'<\cdots< \ell_{K'_Y}'\}$ denote the locations of bad blocks and their neighbours in $X$ and $Y$ respectively.   Let us condition on the block lengths $T_X,T_Y$, $B_1^*,B_2^*$  and the bad-sub-blocks and their neighbours themselves.  Denote this conditioning by
\begin{align*}
\mathcal{F}=\{X \in\A{1}_{X,j+1} , Y\in \A{1}_{Y,j+1},T_X,T_Y,K'_X,K'_Y,\ell_1,\ldots,\ell_{K'_X}, \ell_1',\ldots,\ell_{K'_Y}',\\
X_{\ell_1},\ldots,X_{\ell_{K'_X}}, Y_{\ell_1'},\ldots,Y_{\ell_{K'_Y}'}\}.
\end{align*}


%

Let
$$B_{X,Y}=\{(k,k')\in G_{X}\times G_Y: X_{k}\not\ssc Y_{k'}\}$$
and $N_{X,Y}=|B_{X,Y}|$. Let $\mathcal{S}$ deonte the event $\{N_{X,Y}\leq k_0\}$. We first prove the following lemma.

\begin{lemma}
\label{l:ssmap1}
Let $n_X$ and $n_Y$ denote the number of chunks in $X$ and $Y$ respectively. Fix an entry exit pair of chunks. For concreteness, take $((k,1), (n_{X},k'))\in \mathcal{E}(X,Y)$. Fix $t\in [L_j^{\alpha-1}+2L_j^3+T_X]$ and $t'\in [L_j^{\alpha-1}+2L_j^3+T_Y]$ such that $X_t$ is contained in $C_k^{X}$, $Y_{t'}$ contained in $C_{k'}^Y$ also such that $[t,t+L_j^3]\cap B_X=\emptyset$. Also let $A_{t,t'}$ denote the event that $[t,t+L_j^3]\times [1,L_j^3] \cup [L_{j}^{\alpha-1}+T_{X}+L_j^3,L_{j}^{\alpha-1}+T_{X}+2L_j^3]\times [t'-L_j^3,t']$ is disjoint with $B_{X,Y}$. Set $\tilde{X}=(X_t,X_{t+1},\ldots, X_{L_{j}^{\alpha-1}+T_{X}+2L_j^3})$ and $\tilde{Y}=(Y_{1},Y_2,\ldots, Y_{t'})$, call such a pair $(\tilde{X},\tilde{Y})$ to be a \emph{proper section} of $(X,Y)$. Then there exists an event $S_{t,t'}$ with $\P[S_{t,t'}\mid \cf]\geq 1-L_{j+1}^{-4\beta}$ and such that on $S\cap S_{t,t'}\cap A_{t,t'}$, we have $\tilde{X}\ssco \tilde{Y}$.
\end{lemma}

\begin{proof}
Set $I_1=[t,L_{j}^{\alpha-1}+T_{X}+2L_j^3]\cap \Z$, $I_2=[1,t']\cap \Z$.
By Proposition~\ref{p:assign} we can find $L_j^2$ admissible assignments mappings $\Upsilon_h$ with associated $\tau_h$ of $(I_1,I_2)$ w.r.t. $(B_1^*\cap I_1, B_2^*\cap I_2)$ such that we have $\tau_h(\ell_i)=\tau_1(\ell_i)+h-1$ and $\tau^{-1}_h(\ell_i')=\tau_1^{-1}(\ell_i')-h+1$.
As before we can construct a subset $\mathcal{H}\subset [L_j^2]$ with $|\mathcal{H}|= 10k_0L_j <\lfloor L_j^2/ 36k_0^2\rfloor$ so that for all $i_1\neq i_2$ and $h_1,h_2\in\mathcal{H}$ we have that $\tau_{h_1}(\ell_{i_1})\neq \tau_{h_2}(\ell_{i_2})$ and $\tau_{h_1}^{-1}(\ell_{i_1}')\neq \tau_{h_2}^{-1}(\ell_{i_2}')$, that is that all the positions bad blocks and their neighbours are assigned to are distinct.

Hence we have for all $h\in \mathcal{H}$
\begin{equation}\label{e:SSMapA}
\P[X_{\ell_{i}} \cc Y_{\tau_{h}(\ell_{i})}\mid \cf] \geq \frac12 S_j(X_{\ell_{i}});
\end{equation}
\begin{equation}\label{e:SSMapB}
\P[X_{\tau_{{h}}^{-1}(\ell_{i}')} \cc Y_{\ell_{i}'}\mid \cf] \geq \frac12 S_j(Y_{\ell_{i}'}).
\end{equation}

If $X_{\ell_i}\notin G_j^{\mathbb{X}}$, or, if neither $X_{\ell_i -1}$ nor $X_{\ell_i +1}$ is $\in G_j^{\mathbb{X}}$, let $\cd_{h,i,X}$ denote the event
$$\cd_{h,i,X}=\left\{X_{\ell_{i}} \cc Y_{\tau^{k,k'}_{h}(\ell_{i})}\right\} .$$
If $X_{\ell_i},X_{\ell_i +1}\in G_j^{\mathbb{X}}$ then let
$\cd_{h,i,X}$ denote the event
$$\cd_{h,i,X}=\left\{X_{\ell_{i}} \cs Y_{\tau^{k,k'}_{h}(\ell_{i})}\right\} .$$
If $X_{\ell_i},X_{\ell_i -1}\in G_j^{\mathbb{X}}$ then let
$\cd_{h,i,X}$ denote the event
$$\cd_{h,i,X}=\left\{X_{\ell_{i}} \scc Y_{\tau^{k,k'}_{h}(\ell_{i})}\right\} .$$

Let $\cd_{h,X}$ denote the event
\[
\cd_{h,X}=\bigcap _{i=1}^{K'_X}\cd_{h,i,X}
\]

Let us define the event $\cd_{h,Y}$ similarly and let

\[
\cd_h= \cd_{h,X} \cap \cd_{h,Y}
\]

Finally, let
\[
\cd=\left\{\sum_{h\in \ch} \mathbf{1}_{\cd_h} \geq R^6k_0^510^{2j+20}\right\}.
\]

Conditional on $\cf$, for $h\in\ch$, the $\cd_h$  are independent and by~\eqref{e:SSMapA}, ~\eqref{e:SSMapB} and the recursive estimates ,
\begin{equation}\label{e:SSMapC}
\P[\cd_h \mid \cf]  \geq 2^{-10k_0}3^{4k_0} L_j^{-2/3}.
\end{equation}
Hence using a large deviation estimate for binomial tail probabilities we get,
\begin{equation}\label{e:SSMapD}
\P[\cd \mid \cf]  \geq \P[\mbox{Bin}(10k_0L_j,2^{-10k_0}3^{4k_0} L_j^{-2/3} )\geq R^6k_0^510^{2j+20}\}]\geq  1-  L_{j+1}^{-4\beta}
\end{equation}
for $L_0$ sufficiently large. Now it follows from Lemma \ref{l:admissiblemapexistence} and Lemma \ref{l:ccsssstar} that if $\cd$, $\mathcal{S}$, and $A_{t,t'}$ all holds than $\tilde{X}\ssco \tilde{Y}$. This completes the proof of the lemma.
\end{proof}

Before proving Proposition \ref{l:SSMap}, we need the following lemma bounding the probability of $\mathcal{S}$.
%

\begin{lemma}\label{l:ssfinite}
We have
\begin{equation}\label{e:ssfinite}
\P[\neg S\mid \cf]\leq \frac{1}{3}L_{j+1}^{-3\beta}.
\end{equation}
\end{lemma}

\begin{proof}
Let for $k'\in G_Y$,
\[
V_{k'}^Y=I\left[\left\{\#\left\{k\in G_{X}: X_{k}\not\ssc Y_{k'}\right\}\geq 1 \right\}\right].
\]

It follows from taking a union bound and using the recursive estimates that

\[\P[V_{k'}^Y=1 \mid \cf, X]\leq 2L_j^{\alpha-1-\beta}.
\]

Since $V_{k'}^Y$ are conditionally independent given $X$ and $\cf$, a stochastic domination argument yields
\[
\P[\sum_{k'}V_{k'}^Y \geq k_0^{1/2} \mid X, \cf]\leq P[\mbox{Bin}(2L_j^{\alpha -1}, 2L_j^{\alpha -1-\beta})\geq k_0^{1/2}].
\]

Using Chernoff bound we get

\begin{eqnarray*}
\P[\sum_{k'}V_{k'}^Y \geq k_0^{1/2} \mid \cf , X]&\leq & \exp \left(4L_{j}^{2\alpha -2 -\beta}(\frac{1}{4}k_0^{1/2}L_{j}^{-2\alpha +2 +\beta}-1- \frac{1}{4}k_0^{1/2}L_{j}^{-2\alpha +2 +\beta}\log (\frac{1}{4}k_0^{1/2}L_{j}^{-2\alpha +2 +\beta})\right)\\
&\leq & \exp \left(4L_{j}^{2\alpha -2 -\beta}( -\frac{1}{8}k_0^{1/2}L_{j}^{-2\alpha +2 +\beta}\log (\frac{1}{4}k_0^{1/2}L_{j}^{-2\alpha +2 +\beta})\right)\\
&\leq & \left(\frac{1}{4}k_0^{1/2}L_{j}^{-2\alpha +2 +\beta}\right)^{k_0^{1/2}/2} \leq \frac{1}{6}L_{j+1}^{-3\beta}
\end{eqnarray*}
for $L_0$ large enough since $k_0^{1/2}(\beta +2 -2\alpha)> 6\alpha \beta$.

Removing the conditioning on $X$ we get,

\[
\P[\sum_{k'}V_{k}^Y \geq k_0^{1/2} \mid \cf ]\leq \frac{1}{6}L_{j+1}^{-3\beta}.
\]

Defining $V_k^X$'s similarly we get
\[
\P[\sum_{k}V_{k}^X \geq k_0^{1/2}\mid \cf]\leq \frac{1}{6}L_{j+1}^{-3\beta}.
\]

Since on $\cf$,

\[
\neg \mathcal{S}\subseteq \{\sum_{k}V_{k}^X \geq k_0^{1/2}\}\cup \{\sum_{k}V_{k}^X \geq k_0^{1/2}\},
\]
the lemma follows.
\end{proof}

Now we are ready to prove Proposition \ref{l:SSMap}.

\begin{proof}[Proof of Proposition \ref{l:SSMap}]
Consider the set-up of Lemma \ref{l:ssmap1}. Let $T_{k}$ (resp. $T'_{k'}$) denote the set of indices $t$ (resp. $t'$) such that $X_t$ is contained in $C_{k}^{X}$ (resp. $Y_{t'}$ is contained in $C_{k'}^{Y}$). It is easy to see that there exists $T_{k,*}\subset T_{k}$ (resp. $T'_{k',*}\subset T'_{k'}$) with $|T_{k,*}|\geq (1-10k_0L_j^{-1})|T_{k}|$ (resp. $|T'_{k',*}|\geq (1-10k_0L_j^{-1})|T'_{k'}|$) such that for all $t\in T_{k,*}$ and for all $t'\in T'_{k',*}$, $\tilde{X}$ and $\tilde{Y}$ defined as in Lemma \ref{l:ssmap1} satisfies that $(\tilde{X},\tilde{Y})$ is a \emph{proper section} of $(X,Y)$ and $A_{t,t'}$ holds.

It follows now by taking a union bound over all $t\in T_{k}$, $t'\in T'_{k'}$, and all pairs of entry exit chunks in $\mathcal{E}(X,Y)$ and using Lemma \ref{l:csssconnect} that
%
\begin{equation}\label{e:SSMapE}
\P[X  \ssc Y \mid \cf] \geq 1-\frac{1}{3}L_{j+1}^{-3\beta}-4L_j^{2\alpha}L_{j+1}^{-3\beta}\geq 1-L_{j+1}^{-3\beta}
\end{equation}
for $L_0$ sufficiently large since $\beta >2\alpha$.
Now removing the conditioning we get~\eqref{e:SSMapE2}.
\end{proof}

\section{Good Blocks}\label{s:good}
Now we are ready to prove that a block is good with high probability.

\begin{theorem}
\label{t:goodprobability}
Let $X$ be a $\mathbb{X}$-block at level $(j+1)$. Then $\mathbb{P}(X\in G_{j+1}^{\mathbb{X}})\geq 1-L_{j+1}^{-\delta}$. Similarly for $\mathbb{Y}$-block $Y$ at level $(j+1)$, $\mathbb{P}(Y\in G_{j+1}^{\mathbb{Y}})\geq 1-L_{j+1}^{-\delta}$.
\end{theorem}

\begin{proof}
To avoid repetition, we only prove the theorem for $\mathbb{X}$-blocks.
Let $X$ be a $\mathbb{X}$-block at level $(j+1)$ with length $L_j^{\alpha -1}$

Let the events $A_i, i=1,\ldots 5$ be defined as follows.
\[
A_1=\left\{T_X \leq L_{j}^{5}-2L_{j}^3\right\}.
\]

\[
A_2=\left\{\P[X\cc Y\mid X]\geq \frac{3}{4}+2^{-(j+4)} \right\}.
\]

\[
A_3=\left\{\P[X\cs Y\mid X]\geq \frac{9}{10}+2^{-(j+4)} \right\}.
\]

\[
A_4=\left\{\P[X\scc Y\mid X]\geq \frac{9}{10}+2^{-(j+4)} \right\}.
\]

\[
A_5=\left\{ \P[X\ssc Y\mid X]\geq 1-L_j^{2\beta} \right\}.
\]

From Lemma \ref{l:totalSizeBound} it follows that
$$\P[A_1^c]\leq L_{j+1}^{-3\beta}.$$
From Lemma \ref{l:A1Size} and \ref{l:A1Map} it follows that
$$\P[A_2^c]\leq  L_{j+1}^{-3\beta}.$$

From  \eqref{e:ASSize} and Proposition \ref{l:SCMap}
it follows that
$$\P[A_3^c]\leq  L_{j+1}^{-3\beta},~ \P[A_4^c]\leq  L_{j+1}^{-3\beta}.$$

Using Markov's inequality, it follows from Proposition \ref{l:SSMap}
\begin{eqnarray*}
\P[A_5^c] &=& \P[\P[X\not\ssc Y\mid X]\geq L_{j+1}^{-2\beta}]\\
&\leq & \P[X\not\ssc Y]L_{j+1}^{2\beta}\\
&\leq &\left(\P[X\not\ssc Y, X \in\A{*}_{X,j+1} , Y\in \A{*}_{Y,j+1} ]+\P[X\notin \A{*}_{X,j+1}]+\P[Y\notin \A{*}_{Y,j+1}]\right)L_{j+1}^{2\beta}
\leq  3L_{j+1}^{-\beta}.
\end{eqnarray*}

Putting all these together we get
\[
\P[X\in G_{j+1}^{\mathbb{X}}]\geq \P[\cap_{i=1}^5 A_i]\geq 1-L_{j+1}^{-\delta}
\]
for $L_0$ large enough since $\beta >\delta$.
\end{proof}

%
%
%

\medskip

\textbf{Acknowledgements.} The authors would like to thank Peter Winkler for many useful discussions.

\bibliography{scheduling}
\bibliographystyle{plain}


\end{document}